\documentclass[english, 11pt]{amsart} 
\usepackage{pstricks}
\usepackage{graphics}
\usepackage{cite}
\usepackage{amsmath}
\usepackage{amsfonts}
\usepackage{amssymb}
\usepackage{float}
\usepackage{wrapfig}
\usepackage{epic}
\usepackage{eepic}
\usepackage{ae}
\usepackage{mathrsfs,array,graphicx}
\usepackage[all,ps]{xy}
\usepackage{hyperref}
\usepackage{epstopdf}
\usepackage{note}
\usepackage{a4wide}
\usepackage{std_macros}
\title{Combinatorics of the basic stratum}

\author{Arno Kret}

\begin{document}

\begin{abstract}
We express the cohomology of the basic stratum of some unitary Shimura varieties associated to division algebras in terms of automorphic representations of the group in the Shimura datum.
\end{abstract}

\maketitle

\section*{Introduction}

This article is a sequel to the article \cite{kret1}. We remove an hypothesis from the main Theorem of our previous article. In that article we proved a relation between the $\ell$-adic cohomology of the basic stratum of some simple Shimura varieties and the cohomology of the complex Shimura variety. These simple Shimura varieties are those of Kottwitz considered in his Inventiones article \cite{MR1163241} on the construction of Galois representation. The varieties are associated to certain division algebras over $\Q$ with involution of the second kind; we call such varieties Kottwitz varieties.

When we proved the main theorem of our previous article we assumed (essentially) that the Newton polygon associated to the basic stratum had no integral point other than the begin point and the end point. In this article we solve the resulting combinatorial problems when one removes this simplifying condition from the theorem in case the prime $p$ of reduction is split in the center of the division algebra defining the Kottwitz variety. 

A consequence of our final result is an explicit expression for the zeta function of the basic stratum of Kottwitz's varieties at split primes of good reduction. The expressions are in terms of: (1) Automorphic forms on the group $G$ of the Shimura datum, (2) The determinant of the factor at $\p$ of their associated Galois representations, and (3) Polynomials in $q^\alpha$ of combinatorial nature, associated to certain non-crossing lattice paths in the plane $\Q^2$. 

As an application we deduce a formula for the dimension of the basic stratum. 

\bigskip

\textbf{Acknowledgements:} I thank my thesis advisor Laurent Clozel for his all his help. He gave me many ideas and checked several times the arguments and proofs in this article. I also thank Laurent Fargues (my second advisor), and Olivier Schifmann, Peter Scholze, Alberto Minguez, Marco Deieso, Giovanni Rosso and Vito Mandorino for helpful discussions. 

\setcounter{tocdepth}{1}
\tableofcontents

\section{Notations}

Let $p$ be a prime number and let $F$ be a non-Archimedean local field with residue characteristic equal to $p$. Let $\varpi_F \in \cO_F$ be a prime element, and define $q := \#(\cO_F/\varpi_F)$. We write $G_n$ for the topological group $\Gl_n(F)$, and we write $\cH(G)$ for the Hecke algebra of locally compact constantly supported functions on $G$. We often drop the index $n$ from the notation if confusion is not possible. We call a parabolic subgroup $P$ of $G$ \emph{standard} if it is upper triangular, and we write $P = MN$ for its standard Levi decomposition. We write $K$ for the hyperspecial group $\Gl_n(\cO_F)$ and $\cH_0(G)$ for the Hecke algebra of $G$ with respect to $K$. The group $P_0 \subset G$ is the standard Borel subgroup of $G$, $T$ is the diagonal torus of $G$, and $N_0$ is the group of upper triangular unipotent matrices in $G$. 

We write $\widehat G, \widehat T, \widehat M, \ldots$ for the corresponding complex dual groups, $\widehat G = \Gl_n(\C)$, $\widehat T = (\C^\times)^n$, and so on. If $\pi$ is an unramified representation of some Levi subgroup $M$ of $G$ then we write $\varphi_{M, \pi} \in \widehat M$ for the Hecke matrix of this representation. 

Let $n$ be a positive integer. A \emph{partition} of $n$ is a finite, \emph{non-ordered} list of non-negative numbers whose sum is equal to $n$. A \emph{composition} of $n$ is a finite, \emph{ordered} list of positive numbers whose sum is equal to $n$. Recall that the compositions of $n$ correspond to the standard parabolic subgroups of $G$.

We write $A$ for the ring $\C[X_1^{\pm 1}, X_2^{\pm 1}, \ldots, X_n^{\pm 1}]^{\iS_n}$. The Satake transform $\cS$ provides an isomorphism from $\cH_0(G)$ onto the ring $A$. 

Let $n$ and $\alpha$ be positive integers, and let $s$ be a non-negative integer with $s \leq n$. We call the number $s$ the \emph{signature}, and we call the number $\alpha$ the \emph{degree}. 
The function $f_{n \alpha s} \in \cH_0(G)$ is the spherical function whose Satake transform is 
\begin{align}\label{kottwitzfunction}
q^{\alpha s (n-s)/2} \sum_{\nu \in \iS_n \cdot \mu_s } [\nu]^\alpha = q^{\alpha s (n-s)/2} \sum_{I \subset \{1, \ldots, n\}, \#I = s} \prod_{i \in I} X_i^\alpha \in A.
\end{align}
We put $f_{n \alpha s} = 0$ when $n,\alpha,s \in \Z_{\geq 0}$ are such that $n < s$. We will call $f_{n\alpha s}$ a \emph{simple Kottwitz function}. The \emph{composite Kottwitz functions} $f_{n\alpha \sigma}$ are obtained from partitions $\sigma$ of $s$ as follows. Let $\sigma = (\sigma_1, \sigma_2, \ldots, \sigma_r)$ be a partition of $s$. Then we write $f_{n\alpha \sigma} \in \cH_0(G)$ for the convolution product $f_{n\alpha \sigma_1} * f_{n\alpha \sigma_2} * \cdots * f_{n\alpha \sigma_r} \in \cH_0(G)$. 

We write $\chi_c^G$ for the characteristic function on $G$ of the subset of compact elements. Let $\pi$ be a smooth $G$-representation of finite length and $f$ a locally constant, compactly supported function on $G$. Then we write $\Tr(\chi_c^G f, \pi)$ for the \emph{compact trace} \cite{MR1068388} of $f$ against $\pi$. 

Let $m,m' \in \Z_{\geq 1}$. If $\pi$ (resp. $\pi'$) is a smooth admissible representation of $G_m$ (resp. $G_{m'}$), then we write $\pi \times \pi'$ for the $G_{m + m'}$-representation parabolically induced (unitary induction) from the representation $\pi \otimes \pi'$ of the standard Levi subgroup consisting of two blocks, one of size $m$, and the other one of size $m'$. The tensor product $\pi \otimes \pi'$ in the above formula is taken along the blocks of this Levi subgroup. We write $\cR$ for the direct sum $\bigoplus_{n \in \Z_{\geq 0}} \Groth(G_n)$ with the convention that $G_0$ is the trivial group. The group $G_0$ has one unique irreducible representation $\sigma_0$ (the space $\C$, with trivial action). The operation ``direct sum of representations'' together with the product ``$\times$'' turns the vector space $\cR$ into a commutative $\C$-algebra with $\sigma_0$ as unit element. We call it the \emph{ring of Zelevinsky}. 

The ring of Zelevinsky has an involution $\iota$, called the \emph{Zelevinsky involution}. Aubert \cite{MR1285969} gave a refined definition of this involution, making sense for all reductive groups. The involution is defined by $X^\iota := \sum_{P = MN} \eps_P \Ind_P^G(X_N(\delta_P^{-1/2}))$ for all $X \in \cR$. With `involution' we mean that $\iota$ is an automorphism of the complex algebra $\cR$ and it is of order two: $\iota^2 = \Id_{\cR}$. 

We write $\nu$ for the absolute value morphism from $\Gl_1(F) = F^\times$ to $\C^\times$. By a \emph{segment} $S = \langle x, y \rangle$ we mean a set of numbers $\{x, x+1, \ldots, y\}$ where $x, y \in \Q$ and where we need to explain the conventions in case $y \leq x$. In case $y$ is strictly smaller than $x - 1$, then $\langle x,y \rangle = \emptyset$; in case $x$ is equal to $y$, then the segment $\langle x,y \rangle = \{x\}$ has one element. We have one unusual convention: For $y = x-1$ we define the segment $\langle x, y \rangle$ to be the set $\{\star\}$ of one element containing a distinguishing symbol ``$\star$''. The \emph{length} $\ell(S)$ of a segment $S = \langle x,y \rangle$ is defined to be $y - x + 1$. Thus the segment $\{\star\}$ has length $0$, the segment $\{x\}$ has length $1$, the segment $\{x, x+1\}$ has length $2$, etc. We put $\ell \langle x, y \rangle = -1$ in case $y < x -1$.  

For any segment $\langle x,y \rangle$ with $y \geq x$ we write $\Delta \langle x,y\rangle$ for the unique irreducible quotient of the induced representation $\nu^x \times \nu^{x+1} \times \cdots \times \nu^y$. We define $\Delta\{\star\}$ to be $\sigma_0$ (the one-dimensional representation of the trivial group $\Gl_0(F)$), and we define $\Delta \langle x, y \rangle$ to be $0$ in case $y < x - 1$. For any segment $S$ of non-negative length the object $\Delta S$ is a representation of the group $\Gl_n(F)$, where $n$ is the length of $S$. 

For the standard properties of segments we refer to Zelevinsky's work \cite{MR584084} (cf. \cite{MR689531}), but note that our conventions are slightly different, because we allow rational numbers in the segments and we have the segment $\{\star\}$. We mention that this difference is there only for notational purposes, and that it does not change the mathematics. 

For any finite ordered list of segments $S_1, S_2, \ldots, S_t$ we have the product representation $\pi := (\Delta S_1) \times (\Delta S_2) \times \cdots \times (\Delta S_t)$. Observe that, due to our conventions, in case $S_a = \{\star\}$ for some $a$, then $\Delta S_a$ is the unit in $\cR$, and 
\begin{equation}\label{sterretje}
\pi = (\Delta S_1) \times (\Delta S_2) \times \cdots \times \widehat {(\Delta S_a)} \times \cdots \times (\Delta S_t) \in \cR,
\end{equation}
where the hat means that we leave the corresponding factor out of the product. In case $S_b = \emptyset$ for some index $b$, then we have $\pi = 0$ in $\cR$. 

In the combinatorial part of this article the representations of interest are the Speh representations. We recall their definition here. Let $t, h$ be positive integers such that $n = th$. We define $\Speh(h, t)$ to be the (unique) irreducible quotient of the representation $\St_{G_h} \nu^{\tfrac {t-1}2} \times \cdots \times \St_{G_h} \nu^{\tfrac {1-t}2}$. This representation has $t$ segments, $S_a = \langle x_a, y_a \rangle$, $a = 1, \ldots, t$, where 
$$
x_a = \frac {t-h}2 - (a-1) \quand y_a = \frac {t + h}2 - a. 
$$
Observe that, for each index $a$, we have $\ell S_a = h$. Furthermore, for each index $a < t$, we have $x_{a + 1} = x_a - 1$ and $y_{a + 1} = y_a - 1$. 

If $P = MN \subset G$ is a standard parabolic subgroup of $G$, then we have the spherical functions $\chi_N, \widehat \chi_N$ in $\cH_0(M)$ associated to the acute and obtuse Weyl chambers. We refer to \cite[Eq. (1.1), (1.2)]{kret1} for the precise definition and explicit description of these functions. 

\section{Computation of some compact traces}
In this section we compute the compact traces $\Tr(\chi_c^G f_{n\alpha s}, \pi)$ of the simple Kottwitz functions $f$ on a certain class of representations $\pi$. This class will be sufficiently large to contain all smooth representations that occur in the cohomology of (basic) strata of unitary Shimura varieties at primes of good reduction. 

We will follow the following strategy to compute $\Tr(\chi_c^G f_{n\alpha s}, \pi)$. A semistable representation $\pi$ of $G$ is called \emph{standard} if it is isomorphic to a product of essentially square integrable representations. The computation of the compact trace $\Tr(\chi_c^G f_{n\alpha s}, \pi)$ on a square-integrable representation is easy, and using van Dijk's formula adapted for compact traces \cite[Prop.~1.5]{kret1}, we easily deduce formulas for compact traces on the standard representations. Any semistable irreducible representation $\pi$ may be\footnote{Zelevinsky proved in \cite{MR584084} that the standard representations form a basis of $\cR$ as complex vector space. } (uniquely) written as a sum $\pi = \sum_I c_I \cdot I \in \cR$ where $I$ ranges over the standard representations, and the coefficients $c_I \in \C$ are $0$ for nearly all $I$. We have
$$
\Tr(\chi_c^G f, \pi) = \sum_I c_I \Tr(\chi_c^G f, I).
$$
Thus, there are two steps to compute $\Tr(\chi_c^G f, \pi)$: (Prob1) Know the coefficients $c_I$ and (Prob2) Make the sum $\sum_I c_I \Tr(\chi_c^G f, I)$. The first problem (Prob1) is related to the Kazhdan-Lustzig conjecture\footnote{This conjecture is a Theorem, see \cite[Thm.~8.6.23]{MR2838836}}. The ``Kazhdan-Lustzig Theorem" of Beilinson-Bernstein \cite{MR610137} (and \cite{MR1348671}) interprets the multiplicity of any given irreducible representation $\pi$ in the representation $I$. The Kazhdan-Lustzig Theorem interprets this multiplicity as the dimension of certain intersection cohomology spaces, and also as the value at $q = 1$ of certain Kazhdan-Lustzig polynomials. 

For the irreducible representations $\pi$ contributing to the cohomology of Newton strata of unitary Shimura varieties we will not have to deal with problem (Prob1). The Theorem of Moeglin-Waldspurger \cite{MR1026752} (cf. \cite[Thm.~2.1]{kret1}) for the discrete spectrum of the general linear group implies that these representations must be of a very particular kind (\emph{rigid representations}, cf. [\textit{loc. cit}, p.~14]). Any rigid representation is a product of unramified twists of Speh representations in $\cR$, and therefore we restrict our attention to these Speh representations only. Tadic has solved the first problem (Prob1) for the Speh representations. The coefficients $c_I$ turn out to be $-1, 0$ or $1$ for these representations (precise statement in Theorem~\ref{LMthm}). Therefore, we are mostly concerned with the second problem (Prob2).

\subsection{Tadic's determinantal formula} We recall an important character formula of Tadic for the Speh representations. This formula is a crucial ingredient for our computations.

Let $S_1 = \langle x_1, y_1\rangle , S_2 = \langle x_2, y_2\rangle , \ldots, S_t = \langle x_t, y_t\rangle $ be an ordered list of segments defining a representation of the group $G = \Gl_n(F)$. Let $\iS_t$ be the symmetric group on $\{1, 2, \ldots, t\}$. For any $w \in \iS_t$ we define the number $n^w_a$ to be $y_a - x_{w(a)} + 1$. We have
\begin{equation}\label{urts}
\sum_{a = 1}^k n^w_a = \lhk \sum_{a = 1}^k y_a \rhk - \lhk \sum_{a = 1}^k x_{w(a)} \rhk + k = \lhk \sum_{a = 1}^k y_a \rhk - \lhk \sum_{a = 1}^k x_{a} \rhk + k = \sum_{a=1}^k n_a = n.
\end{equation}
The numbers $n^w_a$ need not be positive. We define $\iS_t' \subset \iS_t$ to be subset consisting of those permutations $w\in \iS_t$ such that the numbers $n_a^w$ are positive or $0$. If the permutation $w$ lies in the subset $\iS_t' \subset \iS_t$, then $(n_a^w)$ is a composition of $n$. Assuming that $w \in \iS_t'$ we will write $P_w = M_w N_w$ for the parabolic subgroup of $G$ corresponding to the composition $(n^w_a)$. 

Let $w \in \iS_t'$. We define the segments $S_1^w := \langle x_{w(1)}, y_1\rangle , S_2^w := \langle x_{w(2)}, y_2\rangle , \ldots, S_t^w := \langle x_{w(t)}, y_t\rangle $. We have $\ell(S_a^w) = n_a^w$. We let $\Delta_w$ be the representation of $M_w$ defined by $(\Delta S_1^w) \otimes \cdots \otimes (\Delta S_t^w)$, where the tensor product is taken along the blocks of $M_w$. The representation $I_w$ is defined to be the product $\Delta S_1^w \times \Delta S_2^w \times \cdots \Delta S_t^w$, \ie it is the (unitary) parabolic induction $\Ind_{P_w}^G \Delta_w$ of $\Delta_w$ to $G$. In case $w \in \iS_t \backslash \iS_t'$ we define both $\Delta_w$ and $I_w$ to be $0$. 

\begin{remark}
It is possible that $S_a^w = \{ \star \}$ for some permutation $w$. In that case the representation $\Delta S_a^w$ is the unit element $\sigma_0$ of $\cR$, and thus can be left out of the product that defined $I_w$ (cf. Equation~\eqref{sterretje}). 
\end{remark}

In these notations we have the following theorem:
 
\begin{theorem}[Tadic]\label{LMthm}
Let $\pi$ be a Speh representation of $G$ and let $S_1 = \langle x_1, y_1 \rangle, S_2 = \langle x_2, y_2 \rangle, \ldots S_t = \langle x_t, y_t\rangle$ be its segments. 
The representation $\pi$ satisfies Tadic's determinantal formula $$\pi = \sum_{w \in \iS_t} \sign(w) I_w.$$ 
\end{theorem} 
\begin{proof}
This Theorem was frst proved by Tadic in \cite{MR1359141} for Speh representations with a difficult argument. Chenevier and Renard simplified the proof and observed that the above expression is a determinant of a matrix with coefficients in Zelevinsky's ring $\cR$. Also Badulescu gave a simpler proof of Theorem~\ref{LMthm} in the note \cite{badulescunote} using the Moeglin-Waldspurger algorithm \cite{MR863522}. Recently Lapid and Minguez \cite[Thm.~1]{LM} extended the formula to the larger class of ladder representations. 
\end{proof}

\begin{remark}
Our formulation of Theorem~\ref{LMthm} is weaker than the theorem proved by the above authors, because we consider only \emph{semistable} Speh representations. (They have a similar statement also for the non semistable Speh/ladder representations.) 
\end{remark}

By the definition of the subset $\iS_t' \subset \iS_t$ we have for all $w \in \iS_t$ that $I_w \neq 0$ if and only if $w \in \iS_t'$, and thus we may as well index over the elements $w \in \iS_t'$ in the sum in the above Theorem. In the cases where the inclusion $\iS_t' \subset \iS_t$ is strict, the subset $\iS_t'$ is practically never a subgroup of $\iS_t$, it will neither be closed under composition nor contain inverses of elements.

\subsection{Lattice paths and the Steinberg representation} In this section we will express the compact trace of the functions $f_{n\alpha s}$ on the Steinberg representation in terms of certain lattice paths in $\Q^2$.

We fix throughout this section a positive integer $\alpha$, called the \emph{degree}. This integer will play only a minor role in the computations of this section as it affects only the weights of the paths. The degree will become more important later. 

Let $A^+$ be the polynomial ring $\C[q^a | a \in \Q]$ of rational, formal powers of the variable $q$. Equivalently, $A^+$ is the complex group ring $\C[\Q^+]$ of the additive group $\Q^+$ underlying $\Q$. A \emph{path} $L$ in $\Q^2$ is a sequence of points $\uv_0, \uv_1, \uv_2, \ldots, \uv_r$ such that $\uv_{i+1} - \uv_i = (1, 0)$ (\emph{east}), or $\uv_{i+1} - \uv_i = (1,1)$ (\emph{north-east}). The starting point of $L$ is $\uv_0$ and the end point is $\uv_r$; the number $r$ is the \emph{length}. An eastward step $(1, 0)$ has weight $1$ and a north-eastward step $(a, b) \to (a+1, b+1)$ has weight $q^{-\alpha \cdot a} \in A^+$. The \emph{weight of the path} $L$ is defined to be the product in $A^+$ of the weights of its steps. 

\begin{remark} We allow paths of length zero; such a path consists of one point $\uv_0$ and no steps. The weight of a path of length $0$ is equal to $1$. The paths of length $0$ correspond to compact traces on the special segments $\{\star\}$ introduced earlier. 
\end{remark}

Let $L$ be a path in $\Q^2$. Connect the starting point $\uv_0$ of $L$ with its end point $\uv_r$ via a straight line $\ell$. Then $L$ is called a \emph{Dyck path} if all of its points $\uv_a$ lie on or below the line $\ell$ in the plane $\Q^2$. The Dyck path is called \emph{strict} if none of its points $\uv_a$ other than the initial  and end point, lies on the line $\ell$. 

Let $\ux, \uy$ be two points in $\Q^2$. Then we write $\Dyck_\strict(\ux, \uy) \in A^+$ for the sum of the weights of all the strict Dyck paths that go from the point $\ux$ to the point $\uy$. We call the polynomial $\Dyck_\strict(\ux, \uy)$ the \emph{strict Dyck polynomial}. There are also non-strict Dyck polynomials $\Dyck(\li x, \li y)$ but we are not concerned with those in this subsection; they are important for the computation of compact traces on the trivial representation.
 
Let $f \in \cH_0(G)$ be a function. We abuse notation and write $\widehat \chi_N \cS(f)$ for the $T$-Satake transform of the function $\widehat \chi_N f^{(P_0)}$. This truncation $\widehat \chi_N f$ of an element $f \in A$ is best understood graphically. 

We first extend the notion of a path slightly to the concept of a graph. A \emph{graph} in $\Q^2$ is a sequence of points $\uv_0, \uv_1, \ldots, \uv_r$ with $\uv_{i+1} - \uv_{i} = (1, e)$, where $e$ is an integer. Thus the paths are those graphs with $e \in \{0, 1\}$ for each of its steps. We define the \emph{weight} of a step $(a, b) \to (a+1, b + e)$ to be $q^{-\alpha \cdot e \cdot a} \in A^+$, and the \emph{weight of a graph} is the product of the weights of its steps. 

To a monomial $X = X_1^{e_1} X_2^{e_2} \cdots X_n^{e_n} \in \C[X_1^{\pm 1}, X_2^{\pm 1}, \ldots, X_n^{\pm 1}]$, with $e_i \in \Z$ and $\sum_{i=1}^n e_i = s$ we associate the graph $\cG_X$ with points 
\begin{equation}\label{thegraph}
\uv_0 := \ell(\tfrac {1-n}2), \quad \uv_i := \uv_0 + \lhk i, e_{n} + e_{n-1} + \ldots + e_{n+1-i} \rhk \in \Q^2, 
\end{equation}
for $i = 1, \ldots, n$. 
If $x \in \Q$, then we write $\ell(x)$ for the point $(x, \tfrac sn x)$ on the line $\ell$. 
Because the sum $\sum_{i=1}^n e_i$ is equal to $s$, the end point of the graph is 
$$
\ell(\tfrac {1-n}2) + (n,s) = \ell(\tfrac {n-1}2 + 1) \in \Q^2.
$$ 
We have $\widehat \chi_{N_0} X = X $ if and only if\footnote{This is true because the fundamental weights $\varpi_{\alpha_i}^G$ of the general linear group are of the form $H_1 + \cdots + H_i - \tfrac in (H_1 + H_2 + \ldots + H_n)$ on $\ia_0$. The statement follows also directly from the conclusion made at Equation (1.10) in the proof of Proposition 1.11 in \cite{kret1}.} 
\begin{equation}\label{thecondition}
e_1 + e_2 + \ldots + e_i > \tfrac sn i,
\end{equation}
for all indices $i < n$ (and if $\widehat \chi_{N_0} X \neq X$, then $\widehat \chi_{N_0} M = 0$). 
The condition in Equation~\eqref{thecondition} is true if and only if the graph defined in Equation~\eqref{thegraph} lies strictly below the straight line $\ell \subset \Q^2$ of slope $\tfrac sn$ going through the origin. Furthermore, the evaluation of $X$ at the point 
\begin{equation}\label{thewweight}
\lhk q^{\tfrac {1-n} 2}, q^{\tfrac {3 - n} 2}, \ldots, q^{\tfrac {n - 1} 2} \rhk \in \widehat T,
\end{equation}
equals the weight of the graph $\cG_X$. 

\begin{remark}
The reader might find it strange that in Equation~\eqref{thegraph} we let the graphs go in the inverse direction. Why did we make this convention? We made this convention because we want graphs that stay below the line $\ell$. If we draw the graphs using the `natural' formula, then we get a graph above the line $\ell$ going from right to left. So why not consider only graphs that stay above $\ell$? Of course this is equivalent, but later, when we compute the graph for the trivial representation, we get get graphs whose `natural' formula stays below $\ell$ and goes from left to right. Thus, either way, we have to invert directions. 
\end{remark}

\begin{lemma}\label{heckeevaluation}
Consider the representation $\pi = \one_T(\delta_{P_0}^{1/2})$ of the group $T$. Let $f$ be a function in the spherical Hecke algebra of $T$. Then the trace of $f$ against $\pi$ is equal to the evaluation of $\cS(f) \in A$ at the point
$$
\lhk q^{\tfrac {1-n} 2}, q^{\tfrac {3 - n} 2}, \ldots, q^{\tfrac {n - 1} 2} \rhk \in \widehat T,
$$
\end{lemma}
\begin{proof}
The character $\delta_{P_0}^{1/2}$ on $T$ is equal to 
$$
T \owns (t_1, t_2, \ldots, t_n) \longmapsto | t_1 |^{\tfrac {n-1}2} |t_2|^{\tfrac {n-3}2} \cdots |t_n|^{\tfrac {n-1}2} \in \C^\times. 
$$
To any (rational) cocharacter $\nu \in X_*(T)$ we may associate the composition $(\delta_{P_0}^{1/2} \circ \nu) \colon F^\times \to T \to \C^\times$. We evaluate this composition at the prime element $\varpi_F \in F^\times$.  Thus we have an element of the set
\begin{equation}\label{identificatiesAA}
\Hom(X_*(T), \C^\times) = \Hom(X^*(\widehat T), \C^\times) = X_*(\widehat T) \otimes_\Z \C^\times = \widehat T(\C), 
\end{equation}
where the last isomorphism is given by
$$
X_*(\widehat T) \otimes_\Z \C^\times \owns \nu \otimes z \longmapsto \nu(z) \in \widehat T(\C). 
$$
We have $T = (F^\times)^n$ and thus we have the standard basis $e_i$ on $X_*(T)$. This corresponds to the standard basis $e_i$ on $X_*(\widehat T)$ via the first two equalities in Equation~\eqref{identificatiesAA}. If we take $\nu = e_i$ in $(\delta_{P_0}^{1/2} \circ \nu)(\varpi_F)$ then we get
$$
(\delta_{P_0}^{1/2} \circ e_i)(\varpi_F) = |\varpi_F|^{\tfrac {n-1}2 - i + 1} = q^{\tfrac {1-n}2 + i - 1}. 
$$ 
This completes the verification.
\end{proof}

Let $\ell \subset \Q^2$ be the line of slope $\tfrac sn$ through $0 \in \Q^2$ that we introduced earlier. We write $\ell(x)$ for the point $(x, \tfrac sn x)$ on $\ell$ if $x \in \Q$.

\begin{lemma}\label{steintrace}
The compact trace $\Tr(\chi_c^G f_{n\alpha s}, \St_G)$ on the Steinberg representation is equal to the polynomial $(-1)^{n-1} q^{ s(n-s)/2} \cdot \Dyck_\strict(\ell(\tfrac {1-n}2), \ell(\tfrac {n-1}2 + 1)) \in A^+$. 
\end{lemma}
\begin{proof}
The proof is a translation of a result that we obtained in our previous article:
\begin{equation}\label{traceonsteinberg}
\Tr(\chi_c^G f, \St_G) = (-1)^{n-1} \Tr \lhk \widehat \chi_{N_0} f_{n\alpha s}, \one_T(\delta_{P_0}^{1/2}) \rhk, 
\end{equation}
(see \cite[Prop.~1.13]{kret1}). NB: We wrote $\delta_{P_0}^{1/2}$ and not $\delta_{P_0}^{-1/2}$; the additional sign is there because the Jacquet module at $P_0$ of the Steinberg representation $\St_G$ is $\one_T(\delta_{P_0})$.

In case $f = f_{n\alpha s}$ then every monomial $X$ occurring in $\cS(f)$ is multiplicity free\footnote{Multiplicity free in the sense that no variable $X_i$ occurs with exponent $e_i > 1$ in $X$.}, and therefore the graph $\cG_X$ is in fact a \emph{path}. The above construction $X \mapsto \cG_X$ provides a bijection between the monomials that occur in $\cS(f)$ and the possible paths that go from the point $\uv_0$ to the point $\uv_r$. Finally $\widehat \chi_{N_0} X \neq 0$ if and only if the corresponding path is a Dyck path (see Equation~\eqref{thecondition}). This completes the proof. 
\end{proof}

Compact traces are compatible with twists:

\begin{lemma}\label{twists}
Let $\chi$ be an unramified character of $F^\times$, $\pi$ a smooth irreducible $G$ representation, and $f_{n\alpha \sigma} \in \cH_0(G)$ a function of Kottwitz. Then $\Tr(\chi_c^G f_{n\alpha \sigma}, \pi \otimes \chi) = \chi(\varpi^{\alpha s}_F) \cdot \Tr(\chi_c^G f, \pi)$.
\end{lemma}
\begin{proof}
\cite[Lem.~1.8]{kret1}. 
\end{proof}

\begin{lemma}\label{tracesteindyck}
Assume that $\pi$ is an essentially square integrable representation of the form $\Delta S$, where $S = \langle x,y\rangle $ is a segment of length $n$. Then
$$
\Tr(\chi_c^G f, \Delta \langle x,y\rangle ) = (-1)^{n-1} \cdot q^{\tfrac {s(n-s)}2} \cdot \Dyck_\strict(\ell(x), \ell(y+1)). 
$$
\end{lemma}
\begin{proof}
The representation $(\Delta S) \otimes \nu^{-x + \tfrac {1-n}2}$ is the Steinberg representation, and so Lemma~\ref{steintrace} applies to it. The result then follows from Lemma~\ref{twists}.
\end{proof}

\subsection{Lattice $t$-paths and standard representations} We describe the compact traces on the standard representations of $G$ using ``$t$-paths''.

Let $t$ be a positive integer. Let $\ux = (\ux_a)$ and $\uy = (\uy_a)$ be two ordered lists of points in $\Q^2$, both of length $t$. A \emph{$t$-path from $\ux$ to $\uy$} is the datum consisting of, for each index $a \in \{1, 2, \ldots, t\}$, a path $L_a$ from the point $\ux_a$ to the point $\uy_a$. A $t$-path $(L_a)$ is called a \emph{Dyck $t$-path} if all the paths $L_a$ are Dyck paths. The Dyck path $(L_a)$ is called \emph{strict} if, for each index $a$, no point $\uv_i$ of $L_a$ other than $\uv_0$ and $\uv_r$ lies on the line $\ell$. The \emph{weight} $\weight(L_a)$ of a $t$-path $(L_a)$ is the product of the weights of the paths $L_a$, where $a$ ranges over the set $\{1, 2, \ldots, t\}$. We extend the definition of the strict Dyck polynomial $\Dyck_\strict(\ux, \uy) \in A^+$ also to $t$-paths: The polynomial $\Dyck_\strict(\ux, \uy) \in A^+$ is by definition the sum of the weights of the strict Dyck $t$-paths from the points $(\ux_a)$ to the points $(\uy_a)$. We have
\begin{equation}\label{dijckpolmayintersect}
\Dyck_\strict(\ux, \uy) = \prod_{a=1}^t \Dyck_\strict(\ux_a, \uy_a) \in A^+. 
\end{equation}

\begin{lemma}\label{traceonstandard}
Let $S_1 = \langle x_1, y_1\rangle , S_2 = \langle x_2, y_2\rangle , \ldots, S_t = \langle x_t, y_t\rangle $ be a list of segments and let $I$ be the representation $(\Delta S_1) \times (\Delta S_2) \times \cdots \times (\Delta S_t)$. Then the compact trace $\Tr(\chi_c^G f_{n\alpha s}, I)$ is equal to $(-1)^{n-t} \dyck_\strict(\ux, \uy)$, where for the indices $a = 1, \ldots, t$ we have $\ux_a := \ell(x_a)$ and $\uy_a := \ell(y_a + 1)$. 
\end{lemma}
\begin{remark}
The sign $(-1)^{n-t}$ is equal to $\eps_{M \cap P_0}$, where $M$ is the standard Levi subgroup of $G$ corresponding to the composition $\sum_{a = 1}^t \ell(n_a)$ of $n$. 
\end{remark}

\begin{proof} Let $P$ be the parabolic subgroup of $G$ corresponding to the composition $n = \sum_{a=1}^t \ell(S_a)$ of $n$. Let $\chi_M^G$ be the characteristic function on $M$ of the subset of elements $m \in M$ such that $\langle \varpi^G_\alpha, H_M(m) \rangle = 0$ for all $\alpha \in \Delta_P$.
By the integration formula of van Dijk for compact traces \cite[Prop.~1.5]{kret1} we have 
\begin{align}\label{traceproducttobe}
\Tr(\chi_c^G f, I) & = \Tr\lhk \chi_c^G f^{(P)}_{n\alpha s}, (\Delta S_1) \times (\Delta S_2) \times \cdots \times (\Delta S_t) \rhk \cr
& = \Tr\lhk \chi_c^M \chi_M^G f^{(P)}_{n\alpha s}, (\Delta S_1) \times (\Delta S_2) \times \cdots \times (\Delta S_t) \rhk. 
\end{align}
In our previous article we proved \cite[Prop.~1.10]{kret1} that the function $\chi_c^G f_{n \alpha s}^{(P)}$ is equal to 
\begin{equation}\label{constantterm}
q^{\alpha C(n_a, s_a)} f_{n\alpha s_1} \otimes f_{n\alpha s_2} \otimes \cdots \otimes f_{n\alpha s_t}, 
\end{equation}
where $s_a := \tfrac {n_a}ns$, and
$$
C(n_a, s_a) := \frac {s(n-s)}2 - \sum_{a = 1}^t \frac {s_a(n_a - s_a)}2. 
$$
The constant term in Equation~\eqref{constantterm} vanishes in case one of the numbers $s_a$ is non-integral. We have $(\chi_c^G f_{n \alpha s}^{(P)})^{(P_0 \cap M)} = \chi_M^G f_{n\alpha s}^{(P_0)}$. Consequently, one may rewrite the trace in Equation~\eqref{traceproducttobe} to the product
$$
q^{\alpha C(n_a, s)} \prod_{a = 1}^t \Tr( \chi_c^{G_{n_a}} f_{n_a \alpha s_a}, \Delta S_a), 
$$
By Lemma~\ref{tracesteindyck} we obtain
$$
q^{\alpha C(n_a, s)} \prod_{a = 1}^t (-1)^{n_a - 1} q^{\tfrac{n_a(n_a - s_a)}2 \alpha} \Dyck_\strict(\ell(x_a), \ell(y_a + 1)).
$$
Note that the condition that $s_a$ is integral precisely corresponds to the condition that the vertical distance between the point $\uy_a$ and the point $\ux_a$ has to be integral before paths can exist. Therefore the expression in this last Equation simplifies to the one stated in the Lemma and the proof is complete.
\end{proof}

\subsection{Non-crossing paths} We express the compact traces on Speh representations in terms of non-crossing lattice paths.

We call a $t$-path $(L_a)$ \emph{crossing} if there exists a couple of indices $a, b$ with $a \neq b$ such that the path $L_a$ has a point $\uv \in \Q^2$ in common with the path $L_b$. There is an important condition:
\begin{enumerate}
 \item[$\bullet$] The point $\uv$ of crossing must appear in the list of points $\uv_{a, i}$ that define $L_a$ and it must also occur in the list of points $\uv_{b, i}$ that define $L_b$. 
\end{enumerate}
(Because we work with rational coordinates, the point of intersection could be a point lying halfway a step of a path (for example). We are ruling out such possibilities.)

\begin{figure}[h]
\begin{center}
\medskip
\setlength{\unitlength}{0.00033333in}%
\begingroup\makeatletter\ifx\SetFigFont\undefined%
\gdef\SetFigFont#1#2#3#4#5{%
  \reset@font\fontsize{#1}{#2pt}%
  \fontfamily{#3}\fontseries{#4}\fontshape{#5}%
  \selectfont}%
\fi\endgroup%
\begin{picture}(7318,3718)(4442,-5520)
{\thinlines
\put(4501,-5461){\circle*{102}}
}%
{\put(5401,-5011){\circle*{102}}
}%
{\put(6301,-4561){\circle*{102}}
}%
{\put(9901,-2761){\circle*{102}}
}%
{\put(10801,-2311){\circle*{102}}
}%
{\put(11701,-1861){\circle*{102}}
}%
{\put(4501,-5461){\line( 2, 1){7200}}
}%
{\multiput(6301,-4561)(120.00000,0.00000){8}{\line( 1, 0){ 60.000}}
\multiput(7201,-4561)(251.53868,251.53868){4}{\line( 1, 1){145.384}}
\multiput(8101,-3661)(120.00000,0.00000){8}{\line( 1, 0){ 60.000}}
\multiput(9001,-3661)(251.53868,251.53868){4}{\line( 1, 1){145.384}}
}%
{\multiput(4501,-5461)(114.28571,0.00000){32}{\line( 1, 0){ 57.143}}
\multiput(8101,-5461)(287.88467,287.88467){13}{\line( 1, 1){145.384}}
}%
{\multiput(5401,-5011)(116.12903,0.00000){16}{\line( 1, 0){ 58.065}}
\multiput(7201,-5011)(304.61603,304.61603){2}{\line( 1, 1){145.384}}
\multiput(7651,-4561)(128.57143,0.00000){4}{\line( 1, 0){ 64.286}}
\multiput(8101,-4561)(304.61603,304.61603){2}{\line( 1, 1){145.384}}
\multiput(8551,-4111)(128.57143,0.00000){4}{\line( 1, 0){ 64.286}}
\multiput(9001,-4111)(275.76934,275.76934){7}{\line( 1, 1){145.384}}
}%
\end{picture}%
\end{center}
\begin{caption}{An example of a 3-path corresponding to the representation $\pi$ of $\Gl_{54}(F)$ defined by the segments $\langle 3, 20\rangle $, $\langle 2, 19\rangle $ and $\langle 1,18\rangle $. We take $s = 27$ and we take the permutation $w = (13) \in \iS_3'$.  
The 3 dots on the lower left hand corner are the points $\ux_1, \ux_2$ and $\ux_3$ in $\Q^2$ respectively; the points $\uy_1, \uy_2$ and $\uy_3$ are in the upper right corner. Observe that this $3$-path is non-strict. 
}
\end{caption}\label{figure1}
\end{figure}

We write $\Dyck^+_{\strict}(\ux, \uy)$ for the sum of the weights of the \emph{non-crossing} strict Dyck $t$-paths. Let $\pi$ be the Speh representation of $G$ associated to the Zelevinsky segments $\langle x_1, y_1\rangle , \langle x_2, y_2\rangle , \ldots, \langle x_t, y_t\rangle $ with $x_1 > x_2 > \ldots > x_t$ and $y_1 > y_2 > \ldots > y_t$. We define the points $\ux_a := \ell(x_a) \in \Q^2$ and $\uy_a := \ell(y_a+1) \in \Q^2$, for $a = 1, 2, \ldots, t$. The group $\iS_t$ acts on the free $\Q^2$-module $\Q^{2t} = (\Q^2)^t$ by sending the $a$-th standard basis vector $e_a \in (\Q^2)^t$ to the basis vector $e_{w(a)} \in (\Q^2)^t$. Thus if we have the vector $\ux \in \Q^{2t}$, then we get the new vector $\ux^w$ whose $a$-th coordinate $\ux^w_a \in \Q^{2t}$ is equal to $w(a)$-th coordinate of the vector $\ux$. 

\begin{remark}
The difference $\tfrac sn \cdot (y_a +1) - \tfrac sn \cdot x_{a}$ need not be integral. In that case there do not exist paths from the point $\ux^w_a \in \Q^2$ to $\uy_a \in \Q^2$.
\end{remark}

Let $\pi$ be a Speh representation of type $(h,t)$. The points $\ux_a \in \Q^2$ and $\uy_a \in \Q^2$ lie on the line $\ell \subset \Q^2$, and the point $\ux_a$ lies on the left of the point $\uy_a$ with horizontal distance $y_a + 1 - x_a = \ell(S_a) = h$. The two lists of points may overlap: There could exist couples of indices $(a, b)$ such that $\ux_a = \uy_b$. All points $\ux_a$ and $\uy_b$ are \emph{distinct} if we have $h \geq t$ (cf. Figure~1). 

Assume $h \geq t$. Then, because all the points $\ux_a$, $\uy_b$ are distinct, there is no permutation $w \in \iS_t$ such that one of the segments $S_a^w = \langle x_{w(a)}, y_a\rangle $ is empty or equal to $\{\star\}$ for some index $a$. In particular we have $\iS_t' = \iS_t$. 

\begin{definition}
To any point $\uv \in \Q^2$ we associate the \emph{invariant} $\rho(\uv) := p_2(\uv) \in \Q/\Z$ where $p_2 \colon \Q^2 \to \Q$ is projection on the second coordinate.
\end{definition}
\begin{remark} The horizontal distance between the point $\ux_b$ and the point $\uy_a$ is integral for all indices. Therefore the invariant of the first coordinate is not of interest. However, the vertical distance is the number $s_a^w = \tfrac sn n_a^w \in \Q$, which certainly need not be integral.
\end{remark}

Using this invariant we define a particular permutation $w_0 \in \iS_t$: 

\begin{definition}\label{wnuldef}
Assume $h \geq t$ and \emph{assume} that for each invariant $\rho \in \Q/\Z$ the number of indices $a$ such that the point $\ux_a$ has invariant $\rho$ is equal to the number of indices $a$ such that the point $\uy_a$ has invariant $\rho$. The element $w_0 \in \iS_t$ is the unique permutation such that for all indices $a, b$ we have
\begin{equation}
\lhk a < b \phantom{\int} \right. \textup{and} \left.\phantom{\int} \rho(\ux_a) = \rho(\ux_b) \rhk 
\Longrightarrow \lhk w_0^{-1}(a) > w_0^{-1}(b) \phantom{\int} \right. \textup{and} \left. \phantom{\int} \rho(\uy_a) = \rho(\uy_b) = \rho(\ux_a) \rhk. 
\end{equation}
\end{definition}

\begin{remark}
Observe that the permutation $w_0$ depends on the integer $s$ because the heights of the points $\ux_a$, $\uy_a$, and therefore also their invariants depend on $s$.
\end{remark}

\begin{remark} If our assumption on the invariants $\rho(\ux_a)$ and $\rho(\uy_a)$ in Definition~\ref{wnuldef} is \emph{not} satisfied, then the permutation $w_0$ cannot exist because it has to induce bijections between sets of different cardinality. 
\end{remark}

One could also define the permutation $w_0 \in \iS_t$ inductively: First the index $w_0^{-1}(t) \in \{1, 2, 3, \ldots, t\}$ is the minimal index $b$ such that the points $\ux_t$ and $\uy_b$ have the same invariant. Next, the index $w_0^{-1}(t-1) \in \{1, 2, 3, \ldots, t\}$ is the minimal index $b$, different from $w^{-1}_0(t)$, such that $\ux_a$ and $\uy_b$ have the same invariant. And so on: $w_0^{-1}(t-i) \in \{1, 2, 3, \ldots, t\}$ is the minimal index $b$ different from the previously chosen indices $w_0^{-1}(t)$, $w^{-1}_0(t-1)$, \ldots, $w_0^{-1}(t-i+1)$, such that the points $\uy_{b}$ and $\ux_{t-i}$ have the same invariant. 

\begin{lemma}\label{combinatoriallemma} Let $\pi$ be a Speh representation with parameters $h, t$ with $h \geq t$. Let $d$ be the greatest common divisor of $n$ and $s$ and write $m$ for the quotient $\tfrac nd$. Define the points $\ux_a := \ell(x_a)$ and $\uy_a := \ell(y_a + 1)$. Let $d$ be the greatest common divisor of $n$ and $s$, and write $m$ for the number $\tfrac nd \in \Z$. The following two statements are equivalent:
\begin{enumerate}
\item[(\textit{i})] for each invariant $\rho \in \Q/\Z$ the number of indices $a$ such that the point $\ux_a$ has invariant $\rho$ is equal to the number of indices $a$ such that the point $\uy_a$ has invariant $\rho$;
\item[(\textit{ii})] $m$ divides $t$ or $m$ divides $h$.
\end{enumerate}
\end{lemma}
\begin{remark}
The number $m$ is the order of the element $\tfrac sn$ in the torsion group $\Q/\Z$. 
\end{remark}

\begin{proof}
We first claim that ``$m | t \Rightarrow (i)$''. We have
\begin{equation}\label{relation}
\rho(\ux_{a + 1}) = \rho(\ux_a) - \tfrac sn \in \Q/\Z
\end{equation}
and the same relation for the points $\uy_a$. Therefore, if $m$ divides $t$, then the possible classes of the points $\ux_a$ are equally distributed over the subset $\tfrac sn\Z/\Z \subset \Q/\Z$, and every invariant occurs precisely $\tfrac tm$ times. The same statement also holds for the points $\uy_a$, and in particular $(i)$ is true. This proves the claim. 

We now claim that ``$m|h \Rightarrow (i)$''. Assume $m|h$. Then the invariants of $\ux_a$ and $\uy_a$ are the same for all indices $a$. Thus $(i)$ is true.

We prove that ``$(m {\not|\,} t \textup{ and }m {\not| \,} h) \Rightarrow ((i) \textup{ is false})$''. 
Assume $m {\not|\,} t$ and $m {\not| \,} h$. We first reduce to the case where $t<m$. Assume $t \geq m$. Consider the elements
\begin{equation}\label{weghalen}
\rho(\ux_1), \rho(\ux_2), \ldots, \rho(\ux_m), \quand \rho(\uy_1), \rho(\uy_2), \ldots, \rho(\uy_m) \in \Q/\Z.
\end{equation}
By Equation~\eqref{relation} every possible class in $\tfrac sn\Z/\Z$ occurs precisely once in both lists. Thus, the truth value of $(i)$ is not affected if we remove the elements $\ux_1, \ux_2, \ldots, \ux_m$ and $\uy_1, \uy_2, \ldots, \uy_m$ from the respective lists. Renumber the indices and repeat this argument until $t < m$. Because we assumed that $t$ did not divide $m$ there remains a positive number of elements in the list $\ux_a$ and $\uy_a$. We renumber so that the indices range from $1$ to $t$. Then we have reduced to the case where $1 \leq t < m$. 

Now look at the two lists
$\rho(\ux_1), \rho(\ux_2), \ldots, \rho(\ux_t)$ and $\rho(\uy_1), \rho(\uy_2), \ldots, \rho(\uy_t)$. 
In both lists every class in $\Q/\Z$ occurs \emph{at most} once. We assumed that $m$ does not divide $h$, and therefore $\rho(\ux_1) \neq \rho(\uy_1)$. If there does not exist an index $b$ such that $\rho(\ux_1) = \rho(\uy_b)$, then (i) is false and we are done. Thus assume $\rho(\ux_1) = \rho(\uy_b)$ for some $1 < b \leq t$. By Equation~\eqref{relation} we then have
$\rho(\uy_{b-1}) = \rho(\uy_b) + \tfrac sn = \rho(\ux_1) + \tfrac sn$. 
The invariant $\xi := \rho(\ux_1) + \tfrac sn \in \Q/\Z$ does not occur in the list $\rho(\ux_1), \rho(\ux_2), \ldots, \rho(\ux_t)$. Thus, we have found an invariant, namely $\xi$, occurring once in the list of invariants of the elements $\uy_a$ and does not occur in the list of invariants of the elements $\ux_a$. This contradicts $(i)$ and completes the proof.
\end{proof}

\begin{theorem}\label{nognietzosimpel} Let $\pi$ be a Speh representation with parameters $h, t$ with $h \geq t$. Let $d$ be the greatest common divisor of $n$ and $s$ and write $m$ for the quotient $\tfrac nd$. Define the points $\ux_a := \ell(x_a)$ and $\uy_a := \ell(y_a + 1)$. The compact trace $\Tr(\chi_c^G f_{n\alpha s}, \pi)$ on $\pi$ is non-zero if and only if $m$ divides $t$ or $m$ divides $h$, and if the compact trace is non-zero, then it is equal to $(-1)^{n-t} \sign(w_0) q^{\tfrac {s(n-s)}2 \alpha} \Dyck^+_{\strict}(\ux^{w_0}, \uy)$, where the permutation $w_0 \in \iS_t$ depends on $s$ and is defined in Definition~\ref{wnuldef}. 
\end{theorem}
\begin{remark}
Perhaps one could extend this Theorem to obtain formulas for compact traces on Ladder representations as considered by Minguez and Lapid in \cite{LM}. 
\end{remark}
\begin{proof} [Proof of Theorem~\ref{nognietzosimpel}]
For a technical reason we assume that $0 < s < n$. In case $s = 0$ we have $f_{n\alpha s} = \one_{\Gl_n(\cO_F)}$. All the elements in $\Gl_n(\cO_F)$ are compact and therefore $\chi_c^G f_{n\alpha 0} = f_{n\alpha 0}$. The compact trace becomes the usual trace and the theorem is easy. A similar argument applies in case $s = n$. Thus we may indeed assume $0 < s < n$. 

By Theorem~\ref{LMthm} the compact trace $\Tr(\chi_c^G f, \pi)$ is equal to the combinatorial sum $\sum_{w \in \iS_t} \sign(w) \Tr(\chi_c^G f, I_w)$ for any Hecke operator $f \in \cH(G)$. We apply it to the Kottwitz functions $f = f_{n\alpha s}$. We have $\iS_t' = \iS_t$ because $h \geq t$. Let $w \in \iS_t$. Recall that van Dijk's formula is also true for truncated traces \cite[Prop.~1.5]{kret1}, and thus for any $w \in \iS_t'$ the trace $\Tr(\chi_c^G f, I_w)$ equals $\Tr(\chi_c^G f^{(P_w)}, \Delta_w)$. Thus we have the formula 
\begin{equation}\label{speh.A}
\Tr(\chi_c^G f, \pi) = \sum_{w \in \iS_t} \sign(w) \cdot \Tr(\chi_c^G f^{(P_w)}, \Delta_w).
\end{equation}
By Lemma~\ref{traceonstandard} we get for $f = f_{n \alpha s}$,
\begin{equation}\label{pathsum}
\Tr(\chi_c^G f, \pi) = q^{\tfrac {s(n-s)}2 \alpha} \sum_{w \in \iS_t} \sign(w) \cdot \eps_{P_0 \cap M_w} \cdot \Dyck_{\strict}(\ux^w, \uy). 
\end{equation}

We apply a standard combinatorial argument\footnote{The \emph{Lindstr\"om-Gessel-Viennot Lemma}. The argument appears in many (almost) equivalent forms in the literature. We learned and essentially copied it from Stanley's book \cite[Thm~2.7.1]{MR2868112}. Note however that, strictly speaking, the Theorem 2.7.1 there does not apply as stated at this point in our argument. In the paragraph that follows we show that Stanley's argument may be adapted so that it does apply to our situation.}. Put the \emph{lexicographical order} $<$ on $\Q^2$:
$$
\forall \uu, \uv \in \Q^2: \quad (\uu < \uv) \Longleftrightarrow (\uu_1 < \uv_1 \textup{ or } (\uu_1 = \uv_1 \textup{ and } \uu_2 < \uv_2)).
$$
Let $(L_a)$ be a strict Dyck $t$-path from the points $\ux^w$ to the points $\uy$, and assume that $(L_a)$ has at least one point of crossing. Let $\uv \in \Q^2$ be the point chosen among the points of crossing which is minimal for the lexicographical order on $\Q^2$. Let $(a, b)$ a couple of different indices, minimal for the lexicographical order on the set of all such couples, such that $\uv$ lies on the path $L_a$ and also on the path $L_b$. We define a new path $L_a'$, defined by following the steps of $L_b$ until the point $\uv$ and then following the steps of the path $L_a$. We define $L_b'$ by following $L_a$ until the point $\uv$ and then continuing the path $L_b$. For the indices $c$ with $c \neq a, b$ we define $L_c' := L_c$. Observe that $(L'_a)$ is a $t$-path from the points $\ux^{(ab)w}$ to the points $\uy$. Furthermore, it is a Dyck path (with respect to this \emph{new} configuration of points), and we have $\weight(L_a) = \weight(L_a')$ because the weight is the product of the weights of the steps, and only the order of the steps has changed in the construction $(L_a) \mapsto (L_a')$. The construction is self-inverse: If we apply the construction to the path $(L_a')$ then we re-obtain $(L_a)$. Both paths $(L_a)$ and $(L_a')$ occur in the sum of Equation~\eqref{pathsum}. The sign $\eps_{P_0 \cap M_w}$ is equal to $(-1)^{n-1}(-1)^{t-1}(-1)^{\# \{ c \in \{1, 2, \ldots, t\} \, |\, \ux_{w(c)} = \uy_c\}}$. By the assumption that $h \geq t$, the points in the list $\ux$ are all different to the points in the list $\uy$, and therefore the sign $\eps_{P_0 \cap M_w}$ equals $(-1)^{n-t}$ (and does not depend on the permutation $w$). The sign of the permutation $w$ is opposite to the sign of $(ab)w$. Consequently, the contributions of the paths $(L_a)$ and $(L_a')$ to Equation~\eqref{pathsum} cancel, and only the non-crossing paths remain in the sum. We find
\begin{equation}\label{pathsumnonint}
\Tr(\chi_c^G f, \pi) = (-1)^{n - t} q^{ \tfrac {s(n-s)}2} \sum_{w \in \iS_t} \sign(w) \cdot \Dyck^+_{\strict}(\ux^w, \uy). 
\end{equation}

We need a second notion of crossing paths, called \emph{topological intersection}. Here we mean that, when the $t$-path $L$ is drawn in the plane $\Q^2$ there is a point $\ux \in \Q^2$ lying on two paths $L_a$, $L_b$ occurring in $L$. Because we allow rational coordinates, topological intersection is not the same as intersection: It is easy to give an example of a $2$-path, which, when drawn in the plane $\Q^2$ has one topological intersection point $\ux \in \Q^2$ but the point $\ux$ does not occur in the lists of points $\uv_{1, 0}, \uv_{1, 1}, \ldots, \uv_{1, r_1}$, $\uv_{2, 0}, \uv_{2, 1}, \ldots, \uv_{2, r_2}$ defining the $2$-path. Such paths are considered non-crossing under our definition, even though they may have topological intersection points\footnote{If one uses the wrong, topological notion of intersection, then the proof breaks at 8 lines below Equation~\eqref{pathsum}: The constructed `path' $(L_c')$ is not a path.}.

We claim that there is at most one permutation $w \in \iS_t$ such that the polynomial $\Dyck^+(\ux^w, \uy)$ is non-zero, and that this permutation is the one we defined in Definition~\ref{wnuldef}. Let $\iS_t''$ be the set of all permutations such that $\Dyck^+(\ux^w, \uy) \neq 0$, and assume that $\iS_t''$ contains an element $w \in \iS_t''$. We first make the following observation:
\begin{enumerate}
\item[(Obs)] To any point $\uv \in \Q^2$ we associated the \emph{invariant} $\rho(\uv) := p_2(\uv) \in \Q/\Z$. The horizontal distance between the point $\ux_{w(a)}$ and the point $\uy_a$ is the number $n_a^w$. The vertical distance is the number $s_a^w = \tfrac sn n_a^w \in \Q$. Because $w \in \iS_t''$ there exists a path from the point $\ux_{w(a)}$ to the point $\uy_{a}$. Consequently $s_a^w$ is integral. This implies that $\rho(\ux_{w(a)}) = \rho(\uy_a)$ for all indices $a$ and in particular the invariant of the point $\ux_{w(a)}$ is independent of $w \in \iS_t''$. 
\end{enumerate}
We show inductively that $w$ is uniquely determined. We start with showing that the index $w^{-1}(t) \in \lbr1, 2, \ldots, t \rbr$ is determined. We claim that $w^{-1}(t) \in \lbr 1, 2, \ldots, t \rbr$ is the minimal index such that the point $\uy_{w^{-1}(t)}$ has the same invariant as $\ux_t$. To see that this claim is true, suppose for a contradiction that it is false, \ie assume the index $w^{-1}(t)$ is \emph{not} minimal. Then there is an index $b$ strictly smaller than $w^{-1}(t)$ such that $\uy_b$ has the same invariant as $\ux_t$. By the observation (Obs) there exists an index $a \neq t$ such that $\ux_a$ has the same invariant as $\ux_t$ and such that $\ux_a$ is connected to $\uy_b$. Draw a picture (see Figure~2) to see that the paths $L_a$ and $L_t$ must intersect topologically. But, by construction, the invariants of $\ux_a$ and $\ux_t$ are the same. Therefore, any topological intersection point of the paths $L_a$ and $L_t$ is a point of crossing. Thus, the paths $L_a$ and $L_b$ are crossing. This is a contradiction, and therefore the claim is true. Thus the value $w^{-1}(t)$ is determined. 

\begin{figure}
\begin{center}
\setlength{\unitlength}{0.00033333in}
\begingroup\makeatletter\ifx\SetFigFont\undefined%
\gdef\SetFigFont#1#2#3#4#5{%
  \reset@font\fontsize{#1}{#2pt}%
  \fontfamily{#3}\fontseries{#4}\fontshape{#5}%
  \selectfont}%
\fi\endgroup%
{\renewcommand{\dashlinestretch}{30}
\begin{picture}(7224,5889)(0,-10)
\put(2037,1587){\circle*{66}}
\put(1137,912){\circle*{66}}
\put(5862,4287){\circle*{66}}
\put(4962,3612){\circle*{66}}
\drawline(12,462)(7212,462)
\drawline(462,12)(462,5862)
\drawline(462,462)(6537,4737)
\drawline(462,462)(6537,4737)
\drawline(2037,1587)(2038,1587)(2040,1586)
	(2043,1584)(2048,1582)(2056,1579)
	(2066,1575)(2079,1569)(2095,1563)
	(2114,1556)(2135,1548)(2160,1539)
	(2186,1530)(2216,1520)(2247,1510)
	(2281,1500)(2316,1491)(2353,1482)
	(2392,1473)(2433,1465)(2475,1458)
	(2518,1452)(2564,1447)(2611,1444)
	(2660,1442)(2711,1442)(2765,1444)
	(2821,1448)(2881,1454)(2943,1464)
	(3009,1476)(3079,1491)(3151,1510)
	(3228,1532)(3306,1558)(3387,1587)
	(3457,1615)(3526,1644)(3593,1675)
	(3657,1705)(3718,1735)(3775,1764)
	(3829,1792)(3878,1818)(3922,1842)
	(3963,1864)(3999,1885)(4032,1904)
	(4061,1921)(4087,1937)(4111,1951)
	(4131,1964)(4150,1976)(4167,1987)
	(4183,1998)(4198,2008)(4212,2018)
	(4226,2029)(4241,2040)(4256,2052)
	(4273,2065)(4291,2080)(4310,2096)
	(4333,2115)(4357,2135)(4384,2159)
	(4415,2185)(4449,2215)(4486,2248)
	(4527,2285)(4572,2326)(4620,2371)
	(4672,2419)(4727,2472)(4784,2528)
	(4843,2587)(4903,2648)(4962,2712)
	(5020,2777)(5076,2841)(5129,2906)
	(5180,2969)(5228,3032)(5273,3092)
	(5315,3152)(5355,3209)(5392,3265)
	(5427,3320)(5460,3374)(5490,3426)
	(5519,3477)(5546,3527)(5572,3576)
	(5596,3624)(5619,3671)(5641,3718)
	(5662,3763)(5682,3808)(5700,3851)
	(5718,3893)(5735,3934)(5750,3974)
	(5765,4012)(5778,4048)(5791,4082)
	(5803,4114)(5813,4143)(5822,4169)
	(5831,4193)(5838,4214)(5844,4232)
	(5849,4247)(5853,4260)(5856,4269)
	(5859,4276)(5860,4281)(5861,4285)
	(5862,4286)(5862,4287)
\drawline(1137,912)(1138,911)(1140,909)
	(1143,906)(1148,902)(1154,897)
	(1161,891)(1170,883)(1181,875)
	(1194,865)(1209,855)(1225,843)
	(1244,832)(1264,820)(1285,808)
	(1309,797)(1334,785)(1361,775)
	(1389,765)(1420,757)(1452,750)
	(1487,744)(1524,740)(1564,738)
	(1607,738)(1653,740)(1702,745)
	(1756,753)(1814,764)(1877,778)
	(1944,797)(2017,819)(2094,846)
	(2176,877)(2262,912)(2328,941)
	(2395,972)(2461,1004)(2527,1036)
	(2591,1068)(2652,1100)(2711,1131)
	(2767,1161)(2820,1190)(2870,1217)
	(2917,1243)(2960,1268)(3001,1291)
	(3039,1313)(3074,1333)(3107,1352)
	(3138,1370)(3167,1387)(3193,1402)
	(3219,1418)(3243,1432)(3267,1446)
	(3290,1460)(3312,1475)(3334,1489)
	(3357,1503)(3380,1518)(3404,1534)
	(3429,1551)(3455,1569)(3483,1589)
	(3513,1610)(3544,1633)(3578,1657)
	(3614,1685)(3653,1714)(3694,1746)
	(3738,1780)(3785,1817)(3834,1857)
	(3886,1900)(3940,1945)(3996,1993)
	(4054,2044)(4112,2096)(4171,2150)
	(4230,2206)(4287,2262)(4356,2333)
	(4421,2403)(4481,2471)(4535,2537)
	(4585,2600)(4631,2661)(4671,2720)
	(4708,2776)(4741,2829)(4770,2881)
	(4796,2931)(4819,2979)(4839,3025)
	(4858,3070)(4874,3113)(4888,3155)
	(4900,3196)(4911,3236)(4920,3275)
	(4928,3312)(4935,3347)(4940,3381)
	(4945,3413)(4949,3443)(4952,3471)
	(4955,3497)(4957,3520)(4959,3540)
	(4960,3558)(4961,3572)(4961,3585)
	(4962,3594)(4962,3601)(4962,3606)
	(4962,3610)(4962,3611)(4962,3612)
\end{picture}
}
\end{center}
\begin{caption}{ The leftmost point $\ux_t$ is connected to the third point $\uy_{w^{-1}(t)}$, and the second point $\ux_a$ is connected to the last point $\uy_b$. Any $2$-path staying below the line $\ell$ must self-intersect topologically. }
\end{caption}
\label{figure2}
\end{figure}

We now look at the index $t- 1$. The point $\ux_{t-1}$ is connected to the point $\uy_{w^{-1}(t-1)}$. We claim that $w^{-1}(t-1) \in \lbr 1, 2, \ldots, t \rbr$ is the minimal index, different from $w^{-1}(t)$, such that $\uy_{w^{-1}(t-1)}$ has the same invariant as $\ux_{t-1}$. The proof of this claim is the same as the one we explained for the index $t$. We may repeat the same argument for the remaining indices $t-2$, $t-3$, etc. Consequently $w$ is uniquely determined by its properties, and equal to the permutation $w_0$ defined in Definition~\ref{wnuldef}.

We proved that if the set $\iS_t''$ is non-empty, then it contains precisely one element, and this element is equal to $w_0$. Therefore, if the compact trace does not vanish, then $m$ must divide $t$ or $m$ divides $h$ by Lemma~\ref{combinatoriallemma}. Inversely, assume that $m$ divides $t$ or $m$ divides $h$. The permutation $w_0 \in \iS_t$ exists by Lemma~\ref{combinatoriallemma}. We claim that $\Dyck^+_{\strict}(\ux^{w_0}, \uy) \neq 0$, so that $w_0 \in \iS_t''$. To prove this, it suffices to construct one non-crossing $t$-path from the points $\ux^{w_0}$ to the points $\uy$. This is easy (see Figure~3): Let $a$ be an index, and write $n_a^{w_0}$ for the horizontal distance between $\ux^{w_0}_a$ and $\uy_a$ and $s_a^{w_0}$ for the vertical distance. The path $L_a$ from $\ux^{w_0}_a$ to $\uy_a$ is defined to be the path taking $n_a^{w_0} - s_a^{w_0}$ horizontal eastward steps, and then $s_a^{w_0}$ diagonal northeastward steps. Then $(L_a)$ is a strict non-crossing $t$-path and therefore $\Dyck^+_{\strict}(\ux^{w_0}, \uy)$ is non-zero. This completes the proof.
\end{proof}

\begin{figure}
\begin{center}
\setlength{\unitlength}{0.00033333in}
\begingroup\makeatletter\ifx\SetFigFont\undefined%
\gdef\SetFigFont#1#2#3#4#5{%
  \reset@font\fontsize{#1}{#2pt}%
  \fontfamily{#3}\fontseries{#4}\fontshape{#5}%
  \selectfont}%
\fi\endgroup%
{\renewcommand{\dashlinestretch}{30}
\begin{picture}(7262,3677)(0,-10)
\put(931,481){\circle*{46}}
\put(1831,931){\circle*{46}}
\put(31,31){\circle*{46}}
\put(7231,3631){\circle*{46}}
\put(6331,3181){\circle*{46}}
\put(5431,2731){\circle*{46}}
\put(2731,1381){\circle*{46}}
\put(4531,2281){\circle*{46}}
\drawline(31,31)(7231,3631)
\dashline{60.000}(7231,3631)(4081,481)(931,481)
\dashline{60.000}(6331,3181)(3181,31)(31,31)
\dashline{60.000}(5431,2731)(4081,1381)(2731,1381)
\dashline{60.000}(4531,2281)(3181,931)(1831,931)
\end{picture}
}
\end{center}
\begin{caption}{An example of a non-crossing $4$-path $(L_a)$ in case $\tfrac sn = \frac 12 \in \Q/\Z$ and $t = 4$. For each $a$, the path $L_a$ first takes $n_a^{w_0} - s_a^{w_0}$ horizontal steps and then $s_a^{w_0}$ vertical steps. Note that paths with the same invariant do not intersect.
}
\end{caption}
\label{figure3}
\end{figure}

\section{A dual formula} 
The argument for Theorem~\ref{nognietzosimpel} extends to the case where $h \leq t$. This computation more complicated, because the permutation $w \in \iS_t$ that contributes to Equation~\eqref{pathsumnonint} is no longer unique and the signs $\eps_{P_0 \cap M_w}$ in Equation~\eqref{pathsum} depend on the contributing permutations $w$ (these signs are independent of $w$ only in case $h \geq t$). We don't reproduce the computation here, because there is a more elegant approach using the duality of Zelevinsky. 

The Zelevinsky dual of a Speh representation with parameters $(h, t)$ is a Speh representation with the role of the parameters inversed, thus of type $(t, h)$. Furthermore, taking the Zelevinsky dual of the formula of Tadic yields a new character formula, now in terms of duals of standard representations. Of course, the Zelevinsky dual of a standard representation is not standard, rather it is an unramified twist of products in $\cR$ of one dimensional representations. Therefore, we compute first the compact trace on the one dimensional representations, then use van Dijk's theorem \cite[prop.~1.5]{kret1} to obtain formulas for products in $\cR$ of one dimensional representations, and finally use the dual of Tadic's formula to compute the compact traces on Speh representations with $h \leq t$ (opposite inequality to Theorem~\ref{nognietzosimpel}). We will then have computed the formula for all Speh representations. This approach seems longer but that is not true: The individual steps we take also appear in an equivalent form in our original computation. 

\subsection{The trivial representation} We compute the compact traces of spherical Hecke operators acting on the trivial representation of $G$. We recall some definitions on roots and convexes from \cite[\S 1]{MR1369904} and \cite[Chap. 1]{waldlabessetwisted}.  

Let $P$ be a standard parabolic subgroup of $G$. Let $A_P$ be the center of $P$. We write $\eps_P = (-1)^{\dim(A_P/A_G)}$. We define $\ia_P := X_*(A_P) \otimes \R$. If $P \subset P'$ then we have $A_{P'} \subset A_P$ and thus an induced map $\ia_{P'} \to \ia_P$. We write $T = A_{P_0}$. We define $\ia_P^{P'}$ to be the quotient of $\ia_P$ by $\ia_{P'}$. We write $\ia_0 = \ia_{P_0}$ and $\ia_0^G = \ia_{P_0}^G$. 

We write $\Delta$ for the set of simple roots of $T$ occurring in the Lie algebra of $N_0$. For each root $\alpha$ in $\Delta$ we have a coroot $\alpha^\vee$ in $\ia_0^G$. We write $\Delta_P \subset \Delta$ for the subset of $\alpha \in \Delta$ acting non-trivially on $A_P$. For $\alpha \in \Delta_P \subset \Delta$ we send the coroot $\alpha^\vee \in \ia_0^G$ to the space $\ia_P^G$ via the canonical surjection $\ia_0^G \surjects \ia_P^G$. The set of these restricted coroots $\alpha^\vee|_{\ia_P^G}$ with $\alpha$ ranging over $\Delta_P$ form a basis of the vector space $\ia_P^G$. By definition the set of fundamental weights $\{ \varpi_\alpha^G \in \ia_P^{G*}\,\,|\,\, \alpha \in \Delta_P\}$ is the basis of $\ia_P^{G*} = \Hom(\ia_P^G , \R)$ dual to the basis $\{\alpha^\vee_{\ia_P^G}\}$ of coroots. Recall that we have the \emph{acute} and \emph{obtuse} Weyl chambers of $G$. The \emph{acute chamber} $\ia_P^{G+}$ is the set of $x \in \ia_P^G$ such that $ \langle \alpha, x \rangle > 0$ for all roots $\alpha \in \Delta_P$. The \emph{obtuse chamber} ${}^+ \ia_P^G$ is the set of $x \in \ia_P^G$ such that we have the inequality $\langle \varpi_\alpha^G, x \rangle > 0$ for all fundamental weights $\varpi_\alpha^G$, associated to $\alpha \in \Delta_P$. We need another chamber, defined by ${}^{\leq} \ia_P^G = \lbr x \in \ia_P^G \,\, | \,\, \forall \alpha \in \Delta_P \, \langle \varpi_\alpha^G, x \rangle \leq 0 \rbr$. We call this chamber the \emph{closed opposite obtuse Weyl chamber}. Let $\tauc$ be the characteristic function on $\ia_P$ of this chamber. Let $H_M \colon M \to \ia_P$ be the \emph{Harish-Chandra} mapping, normalized such that $|\chi(m)|_p = q^{-\langle \chi, H_M(m) \rangle}$ for all rational characters $\chi$ of $M$. We define the function $\xi_c^G$ on $M_0 = T$ to be the composition ${}^\leq \widehat \tau^G_{P_0}\circ (\ia_{P_0} \surjects \ia_{P_0}^G) \circ H_{M_0}$. 

If $f \in \cH_0(G)$ is a function whose Satake transform is the function $h \in A$, then we often abuse notation, and write $\xi_c^G h$ for the Satake transform of the function $\xi_c^G f^{(P_0)}$, and similarly for the functions $\chi_N f$ and $\widehat \chi_N f$ if $f \in \cH_0(M)$. 
 
The following Proposition and proof are valid for any split reductive group $G$ over a non-Archimedean local field. 

\begin{proposition}\label{cttrivial}
 Let $f$ be a function in the Hecke algebra $\cH_0(G)$. The compact trace $\Tr(\chi_c^G f, \one_G)$ is equal to $\Tr\lhk \xi_c^G f^{(P_0)}, \one_T(\delta_{P_0}^{-1/2})\rhk$. 
\end{proposition}
\begin{proof} 
For comfort we prove the proposition under the additional assumption that $G$ is its own derived group. 
We have 
$$
\Tr(\chi_c^{G(\qp)} f, \one) = \sum_{P = MN} \eps_P \Tr(\widehat \chi_N f^{(P)}, \one(\delta_P^{-1/2})). 
$$
Recall that we have the notation $\varphi_{M, \rho} \in \widehat M$ for the Hecke matrix of a representation $\rho$ of $M$. The Hecke matrix $\varphi_{M, \delta_P^{-1/2}}$ is conjugate in $\widehat M$ to the Hecke matrix $\varphi_{T, \delta_P^{-1/2}\delta_{P_0 \cap M}^{-1/2}} = \delta_{T, \delta_{P_0}^{-1/2}}\in \widehat T \subset \widehat M$.
Recall that the Satake transform is defined by the composition of the morphism $f \mapsto f^{(P_0)}$ with the obvious isomorphism $\cH_0(T) \cong \C[X_*(T)]$ (the Satake transformation for $T$). Therefore 
$$
\Tr(\widehat \chi_N f^{(P)}, \one(\delta_P^{-1/2})) = \cS(\widehat \chi_N f^{(P_0)})(\varphi_{T, \delta_{P_0}^{-1/2}}).
$$
Using linearity of the Satake transform we obtain 
$$
\Tr(\chi_c^{G(\qp)} f, \one) = \cS\lhk\sum_{P = MN} \eps_P \widehat \chi_N f^{(P_0)} \rhk(\varphi_{T, \delta_{P_0}^{-1/2}}). 
$$
Thus we have to compute the function $\sum_{P = MN} \eps_P \widehat \chi_N$ on the group $T$. By definition we have 
$$
\widehat \chi_N = \widehat \tau_P^G \circ H_M.  
$$
Let $W_M$ be the rational Weyl group of $T$ in $M$. Let $t \in T$. Then 
$$
H_M(t) = \frac 1 {\# W_M} \sum_{w \in W_M} w H_T(t). 
$$
Thus $\widehat \chi_N(t) = 1$ if and only if
$$
\forall \alpha \in \Delta_P: \quad \sum_{w \in W_M} \langle \varpi_\alpha^G, wH_T(t) \rangle > 0.
$$
We have for all $\alpha \in \Delta_P$ the inequality $\langle \varpi_\alpha^G, H_T(t) \rangle > 0$ 
if and only if we have $\langle \varpi_\alpha^G, wH_T(t) \rangle > 0$ for all $w \in W_M$. Therefore, we have on the group $T$ 
$$
\widehat \chi_N = \widehat \tau_P^G \circ H_T. 
$$
Thus
$$
\sum_{P = MN} \eps_P \widehat \chi_N = \lhk \sum_{P = MN} \eps_P \widehat \tau_P^G \rhk \circ H_T.
$$
By inclusion-exclusion we have 
$$
\sum_{P = MN} \eps_P \widehat \tau_P^G = {}^{\leq} \widehat \tau_P^G.  
$$
This proves the proposition in case $G = G_{\textup{der}}$. It is easy to deduce the statement from the case $G = G_{\textup{der}}$.
\end{proof}

\begin{remark}
Consider the space $I$ of locally constant functions from $G/P_0$ to $\C$, and equip $I$ with the $G$-action through right translations. Then, with an argument similar to the one above, one may compute the compact traces on the irreducible subquotients $V$ of $C$. Recall from Borel and Walach \cite{MR1721403} that these representations are all mutually non-isomorphic and occur with multiplicity one in $I$. Borel and Walach describe the representations $V$ precisely; they are indexed by the standard parabolic subgroups of $G$.
\end{remark}

\subsection{The dual formula} In this subsection we prove the dual version of Theorem~\ref{nognietzosimpel}.

\begin{lemma}\label{traceonstandarddual}
Let $T_1 = \langle u_1, v_1\rangle $, $T_2 = \langle u_2, v_2\rangle $, \ldots, $T_h = \langle u_h, v_h\rangle $ be a list of segments and consider the representation $J := (\Delta T_1)^\iota \times (\Delta T_2)^\iota \times \cdots \times (\Delta T_h)^\iota$. Then $\Tr(\chi_c^G f_{n\alpha s}, \pi)$ is equal to $q^{s(n-s)/2} \Dyck(\uu, \uv)$, where $\uu_a = \ell(u_a)$ and $\uv_a = \ell(v_a + 1)$ for $a = 1,2,\ldots, t$. 
\end{lemma}
\begin{remark}
 Recall that for the compact trace on the Steinberg representation, $\Tr(\chi_c^G f_{n\alpha s}, \St_G)$ we had the sign $\eps_{P_0}$ multiplied with a \emph{strict} Dyck polynomial. In case $n$ and $s$ are coprime, then any Dyck polynomial from the point $\ell(\tfrac{1-n}2)$ to the point $\ell(\tfrac{n-1}2 + 1)$ is strict; consequently the trace on Steinberg and trivial representation differ only by the sign $\eps_{P_0}$.
\end{remark}
\begin{proof}
The proof is the same as the proof for Lemma~\ref{steintrace}, replacing the result in Equation~\eqref{traceonsteinberg} with the result from Proposition~\ref{cttrivial}. However, we repeat the argument for verification purposes (one has to be careful with the signs). 

Assume first that $h = 1$ and that $\pi$ is the trivial representation of $G$. In the previous subsection we proved that 
$$
\Tr\lhk\chi_c^G f_{n\alpha s}, \pi\rhk = \Tr \lhk \xi_c^G f_{n\alpha s}^{(P_0)}, \one_T(\delta_{P_0}^{-1/2}) \rhk.
$$
To a monomial $X = X_1^{e_1} X_2^{e_2} \cdots X_n^{e_n} \in \C[X_1^{\pm 1}, X_2^{\pm 1}, \ldots X_n^{\pm 1}]$ with $e_i \in \Z$ and $\sum_{i=1}^n e_i = s$ we associate the graph $\cG_X$ with points
\begin{equation}\label{newpath}
\uv_0 := \ell(\tfrac {1-n}2), \quad \uv_i := \uv_0 + (i, e_1 + e_2 + \ldots + e_i) \in \Q^2,
\end{equation}
for $i = 1, 2, \ldots, n$. We have $\xi_c^G X = X$ if and only if 
\begin{equation}\label{andereconditie}
e_1 + e_2 + \cdots + e_i \leq \tfrac sn i,
\end{equation}
for all indices $i < n$, and $\xi_c^G X = 0$ otherwise. The evaluation of $X$ at the point 
\begin{equation}\label{thewweight2}
\lhk q^{\tfrac {n-1}2}, q^{\tfrac {n-3}2}, \ldots, q^{\tfrac {1-n}2} \rhk
\end{equation}
equals the weight\footnote{Equation~\eqref{thewweight2} differs from Equation~\eqref{thewweight} by a sign in the exponents. However, observe also that the graph in Equation~\eqref{newpath} is traced in the direction opposite to the graph in Equation~\eqref{thegraph}.} of the graph $\cG_X$. 

The trace of $f_{n\alpha s}$ against the representation $\one_T(\delta_{P_0}^{-1/2})$ is equal to the evaluation of $f_{n\alpha s}$ at the point in Equation~\eqref{thewweight2} (use Lemma~\ref{heckeevaluation} but notice that the signs are different). The monomials $X$ occurring $\cS(f_{n\alpha s})$ yield paths from the point $\ell(\tfrac {1-n}2) \in \Q^2$ to the point $\ell(\tfrac {n-1}2 + 1)$. The condition in Equation~\eqref{andereconditie} is true if and only if the graph $\cG_X$ lies (non-strictly) below the line $\ell$. Therefore we have
$$
\Tr(\chi_c^G f_{n\alpha s}, \one_G) = q^{\tfrac {s(n-s)}2} \Dyck(\ell(\tfrac {1-n}2), \ell(\tfrac {n-1}2 + 1)).
$$
By twisting with the character $\nu^{-x + \tfrac {1-n}2}$ as we did in Lemma~\ref{tracesteindyck} we find
$$
\Tr \lhk \chi_c^G f, (\Delta\langle u, v\rangle)^\iota \rhk = q^{\tfrac {s(n-s)}2} \Dyck(\ell(x), \ell(y + 1)),
$$
for all segments $\langle u, v \rangle$. Finally the argument in Lemma~\ref{traceonstandard} may be repeated to find the compact traces on duals of standard representations as stated in the Lemma. 
\end{proof}

\begin{theorem}\label{nognietzosimpeldual} Let $\pi$ be a Speh representation with parameters $h, t$ with $h \leq t$. Let $d$ be the greatest common divisor of $n$ and $s$ and write $m$ for the quotient $\tfrac nd$. Let $T_a = \langle u_a, v_a \rangle$ be the segments of $\pi^\iota$. Define the points $\uu_a := \ell(u_a)$ and $\uv_a := \ell(v_a + 1)$. The compact trace $\Tr(\chi_c^G f_{n\alpha s}, \pi)$ is non-zero if and only if $m$ divides $h$ or $m$ divides $t$. Assume that the compact trace is non-zero, then it is equal to $\sign(w_0) q^{\tfrac {s(n-s)}2 \alpha} \Dyck^+(\uu^{w_0}, \uv)$, where the permutation $w_0 \in \iS_h$ is defined in Definition~\ref{wnuldef}. 
\end{theorem}
\begin{proof}
Let $\pi^\iota$ be the representation dual to the representation $\pi$. After dualizing the formula of Tadic for $\pi^\iota$ we obtain an expression of the form 
\begin{equation}\label{dualtadic}
\pi = \sum_{w \in \iS_h} \sign(w) I_w^\iota. 
\end{equation}
The involution $\iota$ on $\cR$ commutes with products. Therefore, if $T_1, \ldots, T_h$ are the Zelevinsky segments of the dual representation $\pi^\iota$, then $I_w^\iota$ is equal to $(\Delta T_1)^\iota \times (\Delta T_2)^\iota \times \cdots (\Delta T_k)^\iota$. By Lemma~\ref{traceonstandarddual} we obtain
$$
\Tr(\chi_c^G f_{n\alpha s}, I_w^\iota) = q^{s(n-s)/2} \Dyck(\uu^w, \uy).
$$
A crucial remark is that the points $\uu$ and $\uv$ are all different because we assume that $h \leq t$. Therefore one may repeat the argument in the proof of Theorem~\ref{nognietzosimpel} using the dual formula in Equation~\eqref{dualtadic}; one only has to interchange $t$ with $h$ and every occurrence of the word ``strict Dyck $t$-path'' with ``Dyck $h$-path'', as the paths that describe the compact traces on (products in $\cR$ of) trivial representations are not necessary strict. 
\end{proof}

\section{Return to Shimura varieties}\label{establishmainformula}\label{section.mainproof}

In our article \cite{kret1} we proved a formula for the basic stratum of certain Shimura varieties associated to unitary groups, subject to a technical condition on the Newton polygon of the basic stratum (that it has no non-trivial integral points). In the previous sections we have completely resolved the combinatorial issues that arise if you remove this condition in case $p$ is totally split in the center of the division algebra. We may now essentially repeat the argument from \cite{kret1} to obtain the description of the cohomology if there is no condition on the Newton polygon of the basic stratum. A large part of the argument remains the same, that part will only be sketched and we refer to \cite{kret1} for the details. 

\subsection{Notations and assumptions} Let $\Sh_K/\cO_E \otimes \Z_{(p)}$ be a Kottwitz variety \cite{MR1163241}. Here we have fixed the following long list of notations and assumptions:
\begin{enumerate}
\item Let $D$ be a division algebra over $\Q$;
\item $F$ is the center of $D$,  assume $F$ is a CM field of the form $F = \cK F^+ \subset \li \Q$, where $F^+$ is totally real, and $\cK/\Q$ is quadratic imaginary; 
\item $*$ is an anti-involution on $D$ inducing complex conjugation on $F$;
\item $n \in \Z_{\geq 0}$ is such that $\dim_F(D) = n^2$; 
\item $G$ is the $\Q$-group with $G(R) = \lbr x \in D_R^\times | g^* g \in R^\times\rbr$ for every commutative $\Q$-algebra $R$;
\item $h$ is an algebra morphism $h \colon \C \to D_\R$ such that $h(z)^* = h(\li z)$ for all $z \in \C$;
\item the involution $x \mapsto h(i)^{-1} x^* h(i)$ on $D_\R$ is positive; 
\item $X$ is the $G(\R)$ conjugacy class of the restriction of $h$ to $\C^\times \subset \C$;
\item $\mu \in X_*(G)$ is the restriction of $h\otimes \C \colon \C^\times \times \C^\times \to G(\C)$ to the factor $\C^\times$ of $\C^\times \times \C^\times$ indexed by the identity isomorphism $\C \isomto \C$;
\item $E \subset \li \Q$ is the reflex field of this Shimura datum $(G, X, h^{-1})$;
\item $\xi$ is an (any) irreducible algebraic representation over $\li \Q$ of $G_{\li \Q}$;
\item Let $f_\infty$ be a function at infinity having its stable orbital integrals prescribed by the identities of Kottwitz in \cite{MR1044820}; it can be taken to be (essentially) an 
Euler-poincar\'e function \cite[Lemma~3.2]{MR1163241} (cf. \cite{MR794744}). The function has the following property: Let $\pi_\infty$ be an $(\ig, K_\infty)$-module occurring as the component at infinity of an automorphic representation $\pi$ of $G$. Then the trace of $f_\infty$ against $\pi_\infty$ is equal to the Euler-Poincar\'e characteristic $\sum_{i=0}^\infty N_\infty (-1)^i \dim \uH^i(\ig, K_\infty; \pi_\infty \otimes \xi)$, where $N_\infty$ is a certain explicit constant (cf. \cite[p.~657, Lemma 3.2]{MR1163241}).
\item $p$ is a prime number where $\Sh_K$ has \emph{good reduction} \cite[\S~5]{MR1124982}, and we assume that $p$ is \emph{split} in $\cK/\Q$; 
\item $K \subset G(\Af)$ is a compact open subgroup, small enough that $\Sh_K / \cO_E\otimes \Zp$ is smooth and such that $K$ decomposes as $K^p K_p$ where $K^p$ is a compact open subgroup of $G(\Af^p)$ and $K_p$ is a hyperspecial compact open subgroup of $G(\Qp)$.\item $\nu_p \colon \li \Q \to \lqp$ is a fixed embedding, $\nu_\infty \colon \li \Q \to \C$ is another fixed embedding, the fields $F, F^+, E, \cK$ are all embedded into $\C$;
\item $\p$ is the $E$-prime induced by $\nu_p$;
\item $\fq$ is the residue field of $E$ at the prime $\p$ and $\lfq$ is the residue field of $\li \Q$ at $\nu_p$; for every positive integer $\alpha$, $E_{\p, \alpha} \subset \lqp$ is the unramified extension of $E_\p$ of degree $\alpha$; $\fqa$ is the residue field of $E_{\p, \alpha}$;
\item $\iota \colon B \hookrightarrow \Sh_{K,\fq}$ is the basic stratum \cite{MR2141705} (cf. \cite{MR2074714, MR1485921, MR1411570, MR1124982});
\item $\chi_c^G$ is the characteristic function on $G(\qp)$ of the subset of compact elements (cf. \cite{MR1068388});
\item $\ell$ is a prime number and $\lql$ an algebraic closure of $\Q_\ell$ together with an embedding $\li \Q \subset \lql$;
\item $\cL$ is the $\ell$-adic local system on $\Sh_K/\cO_E \otimes \Z_{(p)}$ associated to the representation $\xi \otimes \lql$ of $G_{\lql}$ \cite[p.~393]{MR1124982};
\item $U \subset G$ is the subgroup of elements with trivial factor of similitudes;
\item for each infinite $F^+$-place $v$, the number $s_v$ is the unique integer $0 \leq s_v \leq \tfrac 12 n$ such that $U(\R) \cong \prod_v U(s_v, n-s_v)$;
\item the embedding $\li \Q \subset \li \Q_p$ induces an action of the group $\Gal(\lqp/\qp)$ on the set of infinite $F^+$-places. For each $\Gal(\lqp/\qp)$-orbit $\wp$ we define the number 
$s_{\wp} := \sum_{v \in \wp} s_v$, and we write $\sigma_\wp$ for the partition $(s_v)_{v \in \wp}$ of the number $s_\wp$;
\item the function $f_\alpha$ is the function of Kottwitz \cite{MR1044820} associated to $\mu$ (cf. \cite[Prop.~3.3]{kret1}).
\end{enumerate}

\begin{remark}
The second condition (2) is particular for our arguments, and does not occur in \cite{MR1163241}. 
\end{remark}

\subsection{The main argument} 
In this article we compute the factors $\Tr(\chi_c^{G(\qp)} f_\alpha, \pi_p)$ occurring in Theorem~\ref{maintheorem} below. We need to introduce two classes of representations:

\begin{definition}
Consider the general linear group $G_n$ over a non-Archimedean local field. Then a representation $\pi$ of $G_n$ is called a (semistable) \emph{rigid representation} if it is equal to a product of the form 
$$
\prod_{a = 1}^k \Speh(x_a, y)(\eps_a) \in \cR,
$$
where $y$ is a divisor of $n$ and $(x_a)$ is a composition of $\tfrac ny$, and $\eps_a$ are unramified unitary characters. 
\end{definition}

\begin{definition}
A representation $\pi$ of the group $G(\qp) = \qp^\times \times \prod_{\wp | p} \Gl_n(F^+_\wp)$ is called a \emph{rigid} representation if for each $F^+$-place $\wp$ above $p$ the component $\pi_\wp$ is a (semistable) rigid representation of $\Gl_n(F^+_\wp)$ in the previous sense:
$$
\pi_\wp = \prod_{a=1}^k \Speh(x_{\wp, a}, y_\wp)(\eps_{\wp, a}) \in \cR,
$$ 
where two additional conditions hold: (1) $y_{\wp} = y_{\wp'}$ for all $\wp, \wp'|p$, and (2) the factor of similitudes $\Qp^\times$ of $G(\qp)$ acts through an unramified character on the space of $\pi$. We write $y := y_\wp$ and call the set of data $(x_{\wp, a}, \eps_{\wp, a}, y)$ the \emph{parameters} of $\pi$. 
\end{definition}

\begin{remark}
Recall that we work in the semistable setting, both notions of rigid representations that we introduced above in the semistable setting also have a natural variant in the non-semistable case.
\end{remark}

\begin{theorem}\label{maintheorem}
Let $\alpha$ be a positive integer. Assume the conditions (1)-(25) from \S 5.1. Then 
\begin{equation}\label{finalformula}
\sum_{i=0}^\infty (-1)^i \Tr(f^{\infty p} \times \Phi^\alpha_\p, \uH_{\et}^i(B_{\lfq}, \iota^*\cL)) = \sum_{\bo{\ \pi \subset \cA(G)} {\pi_p \textup{ rigid}}} 
\Tr(\chi_c^G f_\alpha, \pi_p) \cdot \Tr(f^p, \pi^p).
\end{equation}
\end{theorem}
\begin{remark}
Using recent results obtained with Lapid~\cite{LK} it is possible to extend the above theorem to the other Newton strata. However the result will be combinatorially complicated. We hope to include this result soon. 
\end{remark}
\begin{proof}[Proof of Theorem~\ref{maintheorem}]
Write $T(f^{p}, \alpha)$ for the left hand side of Equation~\eqref{finalformula}. By Proposition 3.4 of \cite{kret1} we have 
\begin{equation}\label{traceonautomforms}
T(f^p, \alpha) = \Tr(\chi_c^G f_\infty f_\alpha f^p, \cA(G)), 
\end{equation}
for all sufficiently large integers $\alpha$. To simplify notations, we write $f := f_\infty f_\alpha f^p$. 

Let $\pi \subset \cA(G)$ be an automorphic representation of $G$ contributing to the trace $\Tr(\chi_c^G f, \cA(G))$. In [\textit{loc. cit}, p. 20] we explained that $\pi$ may be base changed to an automorphic representation $BC(\pi)$ of the algebraic group $\cK^\times \times D^\times$, and that, in turn, $BC(\pi)$ may be send to an automorphic representation $\Pi := JL(BC(\pi))$ of the $\Q$-group $G^+ = \cK^\times \times \Gl_n(F)$. 

The representation $\Pi$ is a \emph{discrete} automorphic representation of the group $G^+(\A)$, and $\Pi$ is semistable at $p$. The classification of Moeglin-Waldspurger implies that $\pi_\wp$ is the irreducible quotient of the induced representation $$\Ind_{P(\A_F)}^{\Gl_n(\A_F)}\lhk \omega |\cdot |^{\tfrac {y-1}2}, \ldots, \omega |\cdot |^{\tfrac {1-y}2 } \rhk,$$ where $P \subset \Gl_n$ is the homogeneous standard parabolic subgroup having $y$ blocks, and each block is of size $n/y$; the inducing representation $\omega$ is a cuspidal automorphic representation of $\Gl_{n/y}(\A_F)$. 

The representation $\Pi$ comes from an automorphic representation of the group $G$ via Jacquet Langlands and base change. Therefore, $\Pi$ is cohomological and conjugate self dual. These properties descend, up to twist by a character, to the representation $\omega$. The Ramanujan conjecture is proved to be true for the representation $\omega$ by the articles \cite{caran, clozelpurity, MR2800722}. Thus the components $\omega_v$ of $\omega$ are \emph{tempered} representations. Note that, of course, the components $\Pi_v$ are \emph{not} tempered if $\Pi$ is not cuspidal. 

An easy computation shows that $\pi_\wp$ is a rigid representation for all $F^+$-places $\wp$ dividing $p$ \cite[Thm.~2.1]{kret1}. This means that there exists a positive divisor $y$ of $n$, a composition $\tfrac ny = \sum_{a = 1}^k x_a$, and unramified unitary characters $\eps_a$ such that
\begin{equation}\label{spehcomponent}
\pi_\wp \cong \Ind_{P(F_{\wp}^+)}^{\Gl_n(F_{\wp}^+)} \bigotimes_{a=1}^r \Speh(x_a, y)(\eps_a), 
\end{equation}
where $P \subset \Gl_n$ is the standard parabolic subgroup corresponding to the composition $(x_a y)$ of $n$, and the tensor product is along the blocks of the standard Levi factor $M$ of $P$. In Equation~\eqref{spehcomponent} the number $y$ is of \emph{global} nature and does not depend on $\wp$. The other data, $k$, $(x_a)$ and $\eps_a$ do depend on the place $\wp$. 
\end{proof}

We work under the condition that $p$ is split in the center $F$ of the algebra $D$. 
Because the prime $p$ is completely split in the extension $F/\Q$ we have by \cite[Prop.~3.3]{kret1} that $$
f_\alpha = \one_{q^{-\alpha}} \otimes \bigotimes_{v \in \Hom(F^+, \R)} f_{n\alpha s_v}^{\Gl_n(\qp)} \in \cH_0(G(\qp)),
$$
where the numbers $s_v$ are the signatures of the unitary group (cf. subsection 1). 
We compute 
\begin{align*}
\Tr(& \chi_c^{G(\qp)} f_\alpha, \pi_p) = \cr
& = \prod_{v \in \Hom(F^+, \R)} \Tr \lhk \chi_c^{\Gl_n(\qp)} f_{n\alpha s_v}, \Ind_{P(\qp)}^{\Gl_n(\qp)} \bigotimes_{a=1}^r \Speh(x_v, y)(\eps_{v, a}) \rhk \cr
& = \lhk \prod_{v \in \Hom(F^+, \R)} \prod_{a=1}^r \eps_{v,a}(q^{-s_v \tfrac {y \cdot x_a} n \alpha}) \rhk \cdot \cr
& \quad \quad \cdot \prod_{v \in \Hom(F^+, \R)} \Tr \lhk \chi_c^{\Gl_n(\qp)} f_{n\alpha s_v}, \Ind_{P(\qp)}^{\Gl_n(\qp)} \bigotimes_{a=1}^r \Speh(x_{v, a}, y) \rhk.
\end{align*}
Write $\zeta_\pi^\alpha \in \C$ for the product $\prod_v \prod_a \eps_a(q^{-s_v \tfrac {y \cdot x_a} n \alpha})$. The polynomial 
\begin{equation}\label{ZZ}
\Tr \lhk \chi_c^{\Gl_n(F^+_{\wp})} f_{n\alpha \sigma_\wp}, \Ind_{P(F^+_{\wp})}^{\Gl_n(F^+_{\wp})} \bigotimes_{a=1}^r \Speh(x_a, y) \rhk \in \C[q^\alpha], 
\end{equation}
is computed in Theorems \ref{nognietzosimpel} and \ref{nognietzosimpeldual} to be a polynomial defined by the weights of certain non-intersecting lattice paths. In particular the trace in Equation~\eqref{ZZ} vanishes unless the number 
\begin{equation}\label{ZZW}
m_{v,a} := \frac {y \cdot x_{\wp, a}}{ \gcd \lhk y\cdot x_{\wp, a}, \tfrac {y\cdot x_{\wp, a}} n s_\wp\rhk } = \frac {y \cdot x_{\wp, a}}{\gcd(n, s_\wp)} 
\end{equation}
is an integer, and divides either $x_{\wp, a}$ or $y$. We make the \emph{assumption} that the compact trace $\Tr(\chi_c^G f_\alpha, \pi_p)$ is non-zero and therefore these divisibility relations are satisfied.

The number $\zeta_\pi^\alpha \in \C$ is determined by the central character $\omega_\pi \colon Z(\A) \to \C^\times$ of the automorphic representation $\pi$ via the Equation:
\begin{equation}\label{thingsarecentral}
\omega_\pi(x)^{\alpha s/n} = \eps_s(q^\alpha) \cdot \prod_{\wp | p} \prod_{a=1}^r \eps_{\wp, a} \lhk q^{-s_\wp \tfrac {y \cdot x_a} n \alpha} \rhk = \zeta_\pi^\alpha,
\end{equation}
where $\eps_s$ is the contribution of the factor of similitudes, and $x \in Z(\A)$ is the following element of the center $Z$ of $G$:
$$
x := (1) \times \left[q, (q_{\wp})_{\wp}\right] \in Z(\A^p)\times \left[ \qp^\times \times F_{\qp}^{+\times} \right] = Z(\A).
$$
The divisibility relations in Equation~\eqref{ZZW} assure that taking the rational power $s/n$ of $\omega_\pi(x)$ on the left hand side makes sense.
 
\begin{remark}
The number $\zeta_\pi$ is a Weil-$q$-number of weight determined by the local system $\cL$, cf. \cite[Eq.~(3.10)]{kret1}. 
\end{remark}

\begin{definition}
We call a rigid representation $\pi_p$ of $G(\qp)$ of \emph{$B$-type} if for all $\wp$, $\gcd(n, s_\wp)$ divides the product $y \cdot x_{\wp, a}$. Furthermore, for each $F^+$-prime $\wp$ and each index $a$, the number $m_{\wp,a}$ divides either $y$ or $x_{\wp, a}$.
\end{definition}

We have proved that only the $B$-type representations contribute to the (alternating sum of the cohomology spaces) of $B$. Let $\pi_p$ be a $B$-type representation of $G(\qp)$. Then we write
$$
\Pol(\pi) \bydef \Tr \lhk \chi_c^{G(\qp)} f_\alpha, \pi_p \rhk \in \C[q^\alpha].
$$ 
We computed this polynomial in the first 4 sections of this article. Explicitely, it is the product over all $\wp$, over all indices $a$ of the polynomial
\begin{equation}\label{theorempol}
\eps \cdot q^{\tfrac {s_\wp (n-s_\wp)}2 \alpha} \cdot \Dyck^+(\ux^{w_0}, \uy),
\end{equation}
where the lists of points $\ux, \uy \in \Q^2$ are defined by:
\begin{enumerate}
\item If $x_a \leq y$, then $\ux$, $\uy$ are of length $x_a$, and for each $b$ we have 
$$
\ux_b := \ell \lhk \frac {x_a-y}2 \rhk, \quand \uy_b := \ell \lhk \frac {x_a + y}2 \rhk,
$$
\item if $x_a \geq y$, then $\ux$, $\uy$ are of length $y$, and for each $b$ we have
$$
\ux_b := \ell \lhk \frac {y-x_a}2 \rhk, \quand \uy_b := \ell \lhk \frac {x_a + y}2 \rhk,
$$
\end{enumerate}
where $\ell \subset \Q^2$ is the line of slope $\tfrac {s_\wp} n$ going through the origin. The notation does not show, but the points $\ux$, $\uy$ and the permutation $w_0$ depend on $\wp$. The symbol $w_{0}$ is a permutation in the group $\iS_{\min(x_a, y)}$ and is determined by Definition~\ref{wnuldef}. The symbol $\eps$ in Equation~\eqref{theorempol} is a sign and is equal to 
\begin{equation}\label{signproduct}
\nu \cdot \sign(w_{0}),
\end{equation}
where the sign $\nu$ is equal to $(-1)^{n-x_a}$ if $x_a \leq n$ and it is equal to $1$ otherwise.

\subsection{Application: A dimension formula}
In our previous article \cite{kret1} we explained that  Formula~\eqref{finalformula} gives a formula for the number of points in $B$ if one takes $f^p = \one_{K^p}$ and $\cL$ equal to the trivial local system. Using this simplified formula we proved in \cite[Prop.~4.2]{kret1} a dimension formula for the basic stratum.  We now extend this result to the Shimura varieties satisfying conditions (1)-(25) from the first subsection, with $p$ completely split in $F^+$. 

Take $f^p = \one_{K^p}$ and $\cL$ in Theorem~\ref{maintheorem} so that the right hand side of Equation~\eqref{finalformula} counts the number of points in $B$ over finite fields. We computed the class of representations at $p$ contributing to this formula. Each representation $\pi_p$ at $p$ contributes with a certain function $P(q^\alpha)$ to the zeta function of $B$. We call the \emph{order} of $\pi_p$ the order of the function $P(q^\alpha)$ (as function in $q^\alpha$). 

\begin{proposition}\label{partialresult}
 The trivial representation $\pi_p = \one$ contributes with the largest order to the right hand side of Equation~\eqref{finalformula}. 
\end{proposition}

\begin{remark}
In the statement of the proposition, we mean `largest order' in the non-strict sense. In general there are multiple representations contributing to the formula with the same order. 
\end{remark}

The order of the trivial representation is easily computed, it is equal to:  
\begin{equation}\label{dimensionformula}
\sum_{\wp|p} \lhk \sum_{v \in \wp} \frac {s_v(1-s_v)}2 + \sum_{j = 0}^{s_\wp - 1} \lceil j \frac n{s_\wp} \rceil \rhk,  
\end{equation}
(cf. \cite[(4.4)]{kret1}). 

\begin{proof}[Proof of Proposition~\ref{partialresult}]
Let $\pi_p$ be a unitary rigid representation. Pick one $\wp | p$. The component $\pi_{\wp}$ is a rigid representation of $G_\wp := \Gl_n(F_\wp^+)$. 

Assume first that $\pi_\wp$ is a Speh representation. We assume that $h \leq t$, so we will work in the dual setting. Treatment of the non-dual case is essentially the same (see Eq.~\eqref{eqAzA} at the end of this argument below). Let $T_1 = \langle u_1, v_1 \rangle$, $T_2 = \langle u_2, v_2 \rangle$, $\ldots$, $T_h = \langle u_h, v_h \rangle$ be the segments of the Zelevinsky dual $\pi_\wp^\iota$ of $\pi_\wp$. By Tadic's formula the compact trace $\Tr(\chi_c^{G_\wp} f_\wp, \pi_{\wp})$ is an alternating sum of compact traces $\Tr(\chi_c^{G_\wp} f_\wp, I^\iota_w)$ on Zelevinsky duals of certain standard representations $I_w$. The traces $\Tr(\chi_c^{G_\wp} f_\wp, I^\iota_w)$ can be described using graphs as we explained in the first section of this article. The intuition is that the closer the graph is to the line $\ell$, the larger its weight is, and we claim that the largest weight is attained by trivial representation. More precisely, we claim that for all $f \in \cH_0(G_\wp)$ and for all permutations $w \in \iS_h$ we have
\begin{equation}\label{ordineq}
\textup{Ord}(\Tr(\chi_c^{G_\wp} f, I^\iota_w)) \leq \textup{Ord}(\Tr(\chi_c^{G_\wp} f, \one_{G_\wp})),
\end{equation}
where with $\textup{Ord}(h) \in \Q$ of an element $h \in A^+$ we mean the largest element $x \in \Q$ such that $q^x$ occurs as a monomial in the expression of $h$ with non-zero coefficient. By Proposition~\ref{cttrivial} and Dijk's integration formula for compact traces
we have
$$
\Tr(\chi_c^{G_\wp} f, I^\iota_w) = q^{\tfrac {s(n-s)} 2}\cdot \sum_{X, \ \xi_c^{G_\wp} \chi_{M_w}^{G_\wp} X \neq 0} c_X \cdot \cG_X(\uu^w, \uv) \in A^+,
$$
where $X$ ranges over the monomials $X \in A$ of the Satake transform $\cS(f)$ of $f$, $c_X \in \C$ is their coefficient and where we should explain the notation $\cG_X(\uu^w, \uv)$. The symbol $\uu$ denotes the list of points $\uu_a := \ell(u_a) \in \Q^2$ for $a = 1, \ldots h$ and the list of points $\uv$ is defined by $\uv_a := \ell(v_a + 1) \in \Q^2$. The symbol $\cG_X$ is the graph of the monomial $X$ as defined in the first section. Recall however that $\cG_X$ is only well-defined up to the definition of its starting point. The representation $I_w^\iota$ is obtained by induction from a one dimensional representation of a standard Levi subgroup $M_w$ of $G_\wp$. Let $(n_a^w)$ be the corresponding composition of $n$, and let $k_w$ be the length of this composition. We cut the graph $\cG_X$ into $k_w$ pieces, the first piece contains the first $n_1^w$ steps of $\cG_X$, the second piece contains the next block of $n_2^w$ steps of $\cG_X$ and so on. Thus instead of one graph $\cG_X$ we now have $k_w$ graphs, $\cG_{X, 1}^w, \cG_{X, 2}^w, \ldots, \cG_{X, k_w}^w$, all well defined up to their starting points. We let the starting point of the graph $\cG_{X, 1}^w$ be $\uu_1^w$, the starting point of the graph $\cG_{X, 2}^w$ is by definition $\uu_2^w$, and so on. Then $\cG_{X, 1}^w$, $\cG_{X, 2}^w$, \ldots are well defined graphs in $\Q^2$, and due to our definition of starting points, we have
\begin{equation}\label{wweightproduct}
\prod_{a=1}^{k_w} \weight (\cG_{X, a}^w) = \Tr \lhk X, (I_w^\iota)_{N_0}(\delta_{P_0}^{-1/2}) \rhk \in A^+.
\end{equation}
The condition $\xi_c^{G_\wp} \chi_{M_w}^{G_\wp} X \neq 0$ on $X$ means precisely that the graphs $\cG_{X,a}^w$ have endpoint equal to $\uv_a$ and that these graphs do not cross, but may touch, the line $\ell$. 

Starting from the monomial $X$ we can also defined a second graph $\cH_X$, such that
$$
\weight(\cH_X) = \Tr \lhk X, \one(\delta_{P_0}^{-1/2}) \rhk \in A^+. 
$$
This graph has starting point $\ux = \ell(\tfrac {1-n}2)$ and end point $\uy = \ell(\tfrac {n-1}2+1)$; the steps of $\cH_X$ are defined by Formula~\eqref{thegraph}. 

We now claim that
\begin{equation}\label{ert123}
\prod_{a=1}^h \weight (\cG_{X, a}^w) \leq \weight(\cH_X) \in A^+,   
\end{equation}
for the obvious meaning of `$\leq$'.

Before we prove the claim, let us first show a simple fact of graphs. Let $\cG$ be any graph in $\Q^2$. Then we have, for any point $(a, b) \in \Q^2$ that
\begin{equation}\label{graph_shift} 
{\textup{Ord}}(\cG + (x, y)) = \textup{Ord}(\cG) - x \cdot \textup{Height}(\cG),
\end{equation} 
where the height of $\cG$, $\textup{Height}(\cG)$, is the vertical distance between the initial point of $\cG$ and its end point. This formula is easily seen to be true: The order $\textup{Ord}(\cG)$ is equal to the sum of $-a\cdot e$ over all diagonal steps $(a, b) \to (a+1, b+ e)$ occurring in the graph $\cG$. Adding the point $(x, y)$ to $\cG$ amounts to changing $-a\cdot e$ to $-(a + x)e$ in the definition of the order of $\cG$. Thus the order of $\cG$ is shifted by the sum, over all diagonal steps $(a, b) \to (a+1, b+e)$, of the value $-xe$. This gives the formula in Equation~\eqref{graph_shift}. 

We now return to the graphs $\cG_X$ and $\cH_X$ introduced earlier. We cut $\cH_X$ into $h$ consecutive graphs. The first graph $\cH_{X,1}$ consists of the first $n_1^w$ steps of $\cH_X$, the second graph $\cH_{X, 2}$ consists of the second block of $n_2^w$ steps of $\cH_X$, and so on. The graphs $\cG_{X, a}$ have the same shape as the graphs $\cH_{X, a}$, but they are shifted (the graphs are constructed starting from the same monomial $X$). Therefore we have the relations:
\begin{equation}\label{shiftgraphsA}
(\forall a):\quad \cH_{X, a} = \cG_{X, a} - \ell(u_{w(a)}) + \ell(\tfrac {1-n}2 + n_1^w + \ldots n_{a-1}^w),
\end{equation}
(we subtract the initial point of $\cG_{X, a}$, and then add the initial point of $\cH_{X, a}$); in the above formula we have the convention that
$$
n_1^w + n_2^w + \ldots + n_{a-1}^w = 0, 
$$ 
in case $a = 1$. Note also that 
$$
\textup{Ord}(\cH_X) = \sum_{a=1}^h \textup{Ord} (\cH_{X, a}), 
$$
and similarly for $\cG_X$. By Equations~\eqref{graph_shift} and~\eqref{shiftgraphsA} we have 
$$
\textup{Ord}(\cH_{X, a}) = \textup{Ord}(\cG_{X, a}) - u_{w(a)} \cdot s_a^w + (\tfrac {1-n} 2 + n_1^w + \ldots + n_{a-1}^w) \cdot s_a^w,
$$
where $s_a^w := n_a^w \cdot \tfrac sn = {\textup{Height}}(\cG_{X, a}) = {\textup{Height}}(\cH_{X, a})$. Thus we have to compute the following expression  
\begin{equation}\label{Z_BB} 
C(w) = \frac sn \sum_{a=1}^h \lhk \frac {1-n}2 + n_1^w + n_2^w + \ldots + n_{a-1}^w - u_{w(a)} \rhk n_a^w. 
\end{equation}
To show that Equation~\eqref{ert123} is true, we show that $C(w) \leq 0$ for all permutations $w$. 

To prove that $C(w) \leq 0$, we may ignore the factor $\tfrac sn$ in the above expression. We prove in two steps that $C(w) \leq 0$ for all $w$. We first determine the permutation $w$ such that the value $C(w)$ is maximal (Step 1). Then we compute for this particular permutation the value $C(w)$, and observe that it is non-positive (Step 2). 

We begin with Step 1. We want to determine $w$ such that $C(w)$ is maximal. Let us first simplify the expression somewhat. The expression $C(w)$ is maximal for $w$ if and only if
\begin{equation}\label{Z_AAC}
\sum_{a=1}^h \lhk n_1^w + n_2^w + \ldots + n_{a-1}^w - u_{w(a)} \rhk n_a^w,
\end{equation}
is maximal. To derive~\eqref{Z_AAC} we used\footnote{See Equation~\eqref{urts}, but note that by duality the roles of $h$ and $t$ are switched.}, that the sum $\sum_{a=1}^t \tfrac {n-1} 2 n_a^w$ equals $n\tfrac {1-n}2$ and therefore this sum does not depend on $w$. (Similar arguments will appear also below.) We have 
\begin{equation}\label{Z_AAA}
n_a^w = \lhk \frac {t+h} 2 - a\rhk - \lhk \frac {h - t} 2 - (w(a) - 1)\rhk + 1 = t - a + w(a) . 
\end{equation}
and
\begin{equation}\label{Z_AAA2}
u_{w(a)} = \frac {t - h} 2 - (w(a) - 1).
\end{equation}
We plug Equations~\eqref{Z_AAA} and~\eqref{Z_AAA2} into Equation~\eqref{Z_AAC} to get
$$
\sum_{a = 1}^h\lhk (t - 1 + w(1)) + \ldots + (t - (a-1) + w(a-1)) 
 - \frac {t-h}2 + (w(a) - 1)  \rhk n_a^w 
$$
As before, this expression is maximal for $w$, if and only if the expression
\begin{equation}\label{Z_AAAA}
\sum_{a=1}^h \lhk w(1) - 1 + w(2) - 2 + \ldots + w(a-1) - (a-1) + w(a)\rhk (t - a + w(a)) 
\end{equation} 
is maximal. Equation~\eqref{Z_AAAA} is maximal for $w$ if and only if the expression
\begin{equation}\label{Z_B}
\sum_{a=1}^h \lhk w(1) - 1 + w(2) - 2 + \ldots + w(a-1) - (a-1) + w(a)\rhk (w(a)  - a) 
\end{equation}  
is maximal. 
We may rewrite~\eqref{Z_B} to 
\begin{align}\label{Z_C}
 \sum_{a=1}^h (w(1) + w(2) + \ldots + w(a))a -\sum_{a=1}^h& (1 + 2 + \ldots + (a-1) ) w(a) 
\end{align}
We rearrange the first sum as follows. Count for each index $a$ the coefficient of $w(a)$ to  get
$$
\sum_{a=1}^h (w(1) + w(2) + \ldots + w(a))a = \sum_{a=1}^h \rho(h+1 - a)  w(a), 
$$
where 
$$
\rho(a) := 1 + 2 + 3 + \ldots + a = \tfrac 12 a (a+1). 
$$ 
Thus~\eqref{Z_C} equals
$$
\sum_{a=1}^h (\rho(t+1-a) - \rho(a-1)) w(a)
$$
The function $\nu(a)$ defined by 
$ \nu(a) = \rho(h+1-a) - \rho(a-1)$,
is strictly decreasing in $a$ because
$$
\nu(a+1) - \nu(a) = -(h+1). 
$$ 
We are looking for $w$ such that
$$
\sum_{a=1}^h \nu(a) \cdot w(a)
$$
is maximal, with $\nu(a)$ a strictly decreasing function for $a \in \{1, 2, \ldots, h\}$. This maximum is attained by the permutation $w$ defined by $a \mapsto h+1 - a$. This completes Step 1.  

We now do Step 2. Thus we have $w(a) = h+1 - a$ for all indices $a \in \{1, 2, \ldots, h\}$. We compute the sum
\begin{equation}\label{Z_AA}
C(w) = \sum_{a=1}^h \lhk \frac {1-n}2 + n_1^w + n_2^w + \ldots + n_{a-1}^w - u_{w(a)} \rhk n_a^w. 
\end{equation}
 We have 
\begin{align}
n_a^w &= y_a - u_{w(a)} + 1 = \lhk \frac {h + t} 2 - a \rhk - \lhk \frac {h - t}2 - ( w(a) - 1) \rhk + 1\cr 
&= t - a + w(a) = t - a + (h+1) - a = t + h+1 - 2a, 
\end{align}
and we have
\begin{equation}
u_{w(a)} = \frac {t - h} 2 - (w(a) - 1) = \frac {t-h}2 - (h - a). 
\end{equation}

Note also that,
$$
n = \sum_{a=1}^h n_a^w = \sum_{a=1}^h t + h+1 - 2a. 
$$
(cf. Equation~\eqref{urts}). Thus, Equation~\eqref{Z_AA} becomes
\begin{align*} 
\sum_{a=1}^h \lhk \frac {1-n}2 + \lhk \sum_{b = 1}^{a-1} t + h + 1 - 2b \rhk - \lhk \frac {t-h} 2 - (h-a) \rhk \rhk \cdot \lhk t + h + 1 - 2a \rhk 
\end{align*}
An easy (but somewhat lengthy) computation shows that this last formula simplifies to $n(h-t)$. By assumption we have $h \leq t$. We conclude that the value in Equation~\eqref{Z_AA} is non-positive, which is what we wanted to show. We have now established the claim in Equation~\eqref{ert123}. 

With the same proof, but using the non-dual instead, one can show that 
\begin{equation}\label{eqAzA}
\textup{Ord}(\Tr(\chi_c^{G_\wp} f, I_w)) \leq \textup{Ord}(\Tr(\chi_c^{G_\wp} f, \St_{G_\wp})),
\end{equation} 
is true for all representation $I_w$ occurring in Tadic's formula for Speh representations $\pi$ with $h \geq t$. Because
$$
\textup{Ord}(\Tr(\chi_c^{G_\wp} f, \St_{G_\wp})) \leq \textup{Ord}(\Tr(\chi_c^{G_\wp} f, \one_{G_\wp})),
$$
the inequality of Equation~\eqref{ordineq} is true for all Speh representations. We leave it to the reader to deduce that Equation~\eqref{ordineq} also holds for products of Speh representations, and also for the rigid representations of $G_\wp$ (with the characters $\eps_a$ trivial). 

We return to the group $G(\qp)$ and the full representation $\pi_p$. 
The compact trace $\Tr(\chi_c^{G(\qp)} f_\alpha, \pi_p)$ is the product of the traces on the components, 
$$
\Tr(\chi_c^{G(\qp)} f_\alpha, \pi_p) = \Tr_{\qp^\times}(f_s, \pi_s) \cdot \prod_{\wp | p} \Tr(\chi_c^{G_\wp} f_\wp, \pi_\wp), 
$$
(the first term in the product is the contribution of the factor of similitudes). We proved that all the terms of this product are bounded by the trace on the trivial representation. This is then also true for the entire product.
\end{proof}

We now deduce a formula for the dimension.

\begin{theorem}
The dimension of the basic stratum $B$ is equal to 
$$
\sum_{\wp|p} \lhk \sum_{v \in \wp} \frac {s_v(1-s_v)}2 + \sum_{j = 0}^{s_\wp - 1} \lceil j \frac n{s_\wp} \rceil \rhk.
$$
\end{theorem}
\begin{proof}
Apply Proposition~\ref{partialresult} and Theorem~\ref{maintheorem} to find 
$$
\dim(B) \leq \sum_{\wp|p} \lhk \sum_{v \in \wp} \frac {s_v(1-s_v)}2 + \sum_{j = 0}^{s_\wp - 1} \lceil j \frac n{s_\wp} \rceil \rhk.
$$
We now prove the opposite inequality. We return to the final formula we found in Theorem~\ref{maintheorem}:
\begin{equation}\label{finalformulaB}
\sum_{i=0}^\infty (-1)^i \Tr(f^{\infty p} \times \Phi^\alpha_\p, \uH_{\et}^i(B_{\lfq}, \iota^*\cL)) = \sum_{\bo{\ \pi \subset \cA(G)} {\pi_p \textup{ rigid}}} 
\Tr(\chi_c^G f_\alpha, \pi_p) \cdot \Tr(f^p, \pi^p).
\end{equation}
We take in this formula $f^p$ and $\cL$ of the following form. Let $p_1$ be a prime number with
\begin{itemize}
\item $p_1$ is different from $\ell, p$; 
\item the group $G$ splits over $\Q_{p_1}$; 
\item the group $K$ splits into a product $K_{p_1} K^{p_1}$ of a hyperspecial group at $p_1$ and a compact open subgroup $K^{p_1} \subset G(\Af^{p_1})$ outside $p_1$.
\end{itemize} 
We take 
\begin{itemize}
\item $f^{p p_1} = \one_{K^{p p_1}}$;
\item $f_{p_1}$ is an arbitrary $K_{p_1}$-spherical function;
\item $\cL = \lql$ (the trivial local system).
\end{itemize}
There exist only a finite number of representations $\pi_{p_1}$ contributing to Equation~\eqref{finalformulaB}, and one of these representations is the trivial representation. Thus we may find a spherical Hecke operator $f_{p_1} \in \cH(G(\Q_{p_1}))$ such that
$$
\Tr(f_{p_1}, \pi_{p_1}) = \begin{cases}
1 & \pi_{p_1} \cong \one_{G(\Q_{p_1})} \cr
0 & \textup{otherwise,}
\end{cases}
$$ 
for all representations $\pi_{p_1}$ occurring in Equation~\eqref{finalformulaB}. We consider the Hecke operator $f^p := \one_{K^{p p_1}} \otimes f_{p_1}$ in Equation~\eqref{finalformulaB}. By construction, any automorphic representation $\pi \subset \cA(G)$ contributing to Equation~\eqref{finalformulaB} has $\pi_{p_1} \cong \one_{G(\Q_{p_1})}$. By a strong approximation argument, the representation $\pi$ is one dimensional\footnote{See for example Lemma 3.6 in our article~\cite{kret1}, although this result is of course well known.}, and in particular Abelian.  Consequently, at the prime $p \neq p_1$,  the representation $\pi_p$ is a twist of $\one_{G(\qp)}$ by an unramified character $\chi_p$. Because the representation $\xi$ at infinity is trivial, the character $\chi_p$ is of finite order. Therefore there exists an integer $r>0$ such that, whenever $r$ divides $\alpha$, we have 
$$
\Tr(\chi_c^{G(\qp)} f_\alpha, \pi_p) = \Tr(\chi_c^{G(\qp)} f_\alpha, \one), 
$$ 
for all representations $\pi_p$ contributing to Equation~\eqref{finalformulaB}. From now on we consider only $\alpha$ such that $r|\alpha$. The right hand side of Equation~\eqref{finalformulaB} simplifies to 
$$
C \cdot \Tr(\chi_c^{G(\qp)} f_\alpha, \one), 
$$ 
where $C$ is some non-zero constant. Thus for our choice of $f^{\infty p}$ the trace 
\begin{equation}\label{growswithorder}
\sum_{i=0}^\infty (-1)^i \Tr(f^{\infty p} \times \Phi^\alpha_\p, \uH_{\et}^i(B_{\lfq}, \lql)) 
\end{equation}
grows with the order of the trivial representation. View $\sum_{i=0}^\infty (-1)^i \uH_{\et}^i(B_{\lfq}, \lql)$ as a virtual representation of the group  $(\Phi_{\p}^{r})^{\Z}$, and write it as a linear combination of the characters of this group. The character of highest order occurring in this expression determines the dimension of the variety $B$. By the conclusion in Equation~\eqref{growswithorder} there occurs a character whose order is at least  $\textup{Ord}(\Tr(\chi_c^{G(\qp)} f_\alpha, \one))$. This means that
$$
\dim(B) \geq \sum_{\wp|p} \lhk \sum_{v \in \wp} \frac {s_v(1-s_v)}2 + \sum_{j = 0}^{s_\wp - 1} \lceil j \frac n{s_\wp} \rceil \rhk.
$$
This completes the proof of the Theorem. 
\end{proof}

\begin{remark}
The above formula confirms the conjecture for the dimension of the basic stratum specialized to the cases we consider. See for example~\cite{MR2252119}.
\end{remark}

\subsection{Application: Vanishing of the cohomology}

In our previous article~\cite{kret1} we assumed that the signatures $s_\wp$ are coprime to the number $n$.  Under these conditions the cohomology of the basic stratum is very simple: Locally at the prime $p$, only the trivial representation and (essentially) the Steinberg representation contribute to Expression~\eqref{finalformula}. In fact this is true in a larger class of cases:

\begin{corollary}
Assume there is one $F^+$-place $\wp$ above $p$ such that $s_\wp$ is coprime to $n$. Then only the Steinberg representation and the trivial representation contribute to the formula in Equation~\eqref{finalformula}.
\end{corollary}
\begin{proof}
This follows directly form the definition of rigid representation of the group $G(\qp)$. 
\end{proof}

\begin{remark}
In the article \cite{kret1} we assumed that, for all $\wp$, the number $s_\wp$ is coprime to $n$ or $s_\wp$ is equal to $0$ or $n$. Only under this larger assumption the compact trace on the Steinberg representation coincides with the compact trace on the trivial representation (up to sign), just as in \cite{kret1}. In the above Corollary this need not be the case.
\end{remark}

\subsection{Application: Euler-Poincar\'e characteristics} Finally we have a remark on the Euler-Poincar\'e characteristic of the variety $B$. The evaluation at $q = 1$ of our formula gives the expression of the Euler-Poincar\'e characteristic. Thus to compute the Euler-Poincar\'e characteristic we get the combinatorial problem to compute, apart from dimensions of spaces of automorphic forms, the \emph{number} of non-intersecting Dyck paths. This problem has been considered in an equivalent forms in the literature; a good starting point are the books of Stanley~\cite{MR2868112} and the references therein. 

\section{Examples}

We end this article with some examples. Let us first explain why we need the condition that $p$ splits completely in the center of $D$. 

\subsection{Products of simple Kottwitz functions} To study the reduction modulo $p$ of unitary Shimura varieties, the simple Kottwitz functions $f_{n\alpha s}$ as we defined them in Equation~\eqref{kottwitzfunction} are \emph{not} enough. These functions count only points of unitary Shimura varieties if the group $G$ of the Shimura datum is of the following kind. Consider a unitary Shimura variety associated to a division algebra $D$ as in the previous section. Let $U$ in $G$ be the subgroup of elements whose factor of similitudes is equal to one. Then $U$ is a unitary group and $U(\R)$ is isomorphic to a product of standard unitary groups $U(p_\tau, q_\tau)$ with $\tau$ ranging over the infinite places of the maximal totally real subfield $F^+$ of the center $F$ of $D$. The function $f_{n\alpha s}$ counts points on the reduction of $\Sh_K$ modulo $p$ if we have $p_\tau = 0$ or $q_\tau = 0$ for all $F^+$-places $\tau$, but with one $F^+$-place excluded. For the excluded $F^+$-place $\tau_0$ we must have $p_{\tau_0} = s$ or $p_{\tau_0} = n-s$. For unitary Shimura varieties with several non-zero signatures at infinity, one will need to consider products of the functions $f_{n\alpha s}$ for several different values of $s$. 

\begin{remark}
Compact traces do \emph{not} commute with products of Hecke operators. 
\end{remark}

\begin{example} Let us assume that there are two infinite $F^+$-places $\tau_0, \tau_1$ with $p_{\tau_0} = p_{\tau_1} = 1$ and that $p_\tau = 0$ for all other $\tau$. Choose embeddings $\C \supset \li \Q \subset \lqp$, so the group $\Gal(\lqp/\qp)$ acts on the set of infinite places of $F^+$. Assume the places $\tau_0$ and $\tau_1$ lie in the same $\Gal(\lqp/\qp)$-orbit and assume $\alpha$ is sufficiently divisible such that the $E_{\p,\alpha}$-algebra $F^+ \otimes E_{\p,\alpha}$ is split. Then the function counting points in the set $\# \Sh_K(\fqa)$ is (essentially) the convolution product $f = f_{n\alpha 1} * f_{n\alpha 1} \in \cH_0(\Gl_n(F))$, where $F$ is some finite extension of $\Qp$. An easy computation shows that $f = 2q^\alpha f_{n\alpha 2} + f_{n(2\alpha)1}$, and therefore $\Tr(\chi_c^G f, \one_G) = 2(q + q^1 + \ldots + q^{\alpha \lfloor \frac n2 \rfloor}) + 1$ (use the Example in \cite[p.~14]{kret1}). Consequently the number of points in the basic stratum over the field $\F_{q^\alpha}$ is the product of the above polynomial times a cohomological expression depending only on the class of the degree $\alpha$ in the group $\Z/h\Z$, where $h$ is related to the class number of the cocenter of $G$ \cite[Cor.~4.1]{kret1}. In particular the variety is of dimension $\lfloor \frac n2 \rfloor$ in this case. If we assume instead that $\tau_0$ and $\tau_1$ lied in a \emph{different} $\Gal(\lqp/\qp)$-orbit, then the basic stratum of $\Sh_K$ is a \emph{finite} variety. Whether or not $\tau_0$ and $\tau_1$ lie in the same $\Gal(\lqp/\qp)$-orbit is a condition on how the prime $p$ decomposes as a product of prime ideals in the ring of integers $\cO_{F^+}$ of $F^+$. Thus, roughly speaking\footnote{To obtain the precise statement use the discussion at Equation (3.1) of [\textit{loc. cit}].}, the form of the function $\alpha \mapsto \Tr (\chi_c^{G(\qp)} f, \one_{G(\qp)})$ depends only on two pieces of information: (1) The signatures of the unitary group at infinity, and (2) how the prime $p$ decomposes in $F^+$. 
\end{example}

\subsection{Two different prime factors} Assume $F^+$ is of degree $2$ over $\Q$ and $n$ is a product of two primes $x,y$ with $x < y$. Let $U \subset G$ be the subgroup of elements whose factor of similitudes is trivial. We assume $U(\R)$ is isomorphic to $U(x, n-x)(\R) \times U(y, n-y)(\R)$. The reflex field $E$ of the Shimura datum coincides with the field $F$. 

There are two cases to consider, either the prime $p$ where we reduce $\Sh_K$ splits in $F^+$ or $p$ is inert (but unramified). Assume that $p$ splits, then $G(\qp) = \Qp^\times \times \Gl_n(\Qp) \times \Gl_n(\Qp)$. Recall that we picked an embedding $\nu_p \colon \li \Q \to \lqp$. Therefore the factors of the product $\Gl_n(\qp) \times \Gl_n(\qp)$ are ordered: the embbeding $\nu_p$ identifies the two $F^+$-places $\tau_1, \tau_2$ at infinity with the two $F^+$-places $\wp_1, \wp_2$ above $p$. Via the isomorphism $U(\R) \cong U(x, n-x)(\R) \times U(y, n-y)(\R)$ we associate to $\tau_1, \tau_2$ (and thus to $\wp_1, \wp_2$) a signature equal to $x$ or $y$. Assume that $\wp_1$ (and $\tau_1$) correspond to $x$ and $\wp_2$ (and $\tau_2$) correspond to $y$. Similarly, the first factor of the group $\Gl_n(\qp) \times \Gl_n(\qp)$ corresponds to $\wp_1$ and the second factor corresponds to $\wp_2$.

The $B$-type representations of $G(\qp)$ are the representations contributing to the cohomology of the basic stratum. Ignoring the factor of similitudes, the $B$-type representations of $G(\qp)/\Qp^\times = \Gl_n(\Qp) \times \Gl_n(\qp)$ are:
\begin{equation}\label{repA}
\speh(x, y)(\eps) \otimes \prod_{a=1}^k \Speh(x_a, y)(\eps_a),
\end{equation}
\begin{equation}\label{repB}
\prod_{a=1}^k \Speh(y_a, x)(\eps_a) \otimes \Speh(x, y)(\eps),
\end{equation}
\begin{equation}\label{repC}
\St_G(\eps) \otimes \St_G(\eps'),
\end{equation}
\begin{equation}\label{repD}
\one_G(\eps) \otimes \one_G(\eps'),
\end{equation}
where, in these equations the number $k$ can, a priori, be any positive number. In Equation~\eqref{repA}, the symbol $(x_a)$ ranges over the compositions of the prime $x$ and in Equation~\eqref{repB}, the symbol $(y_a)$ ranges over the compositions of the prime $y$. The symbols $\eps, \eps', \eps_a$ denote arbitrary, unrelated, unramified unitary characters. 

\subsection{Some explicit polynomials} We specialize our first example further, and assume that $x = 2$ and $y = 3$, so $U(\R) \cong U(2, 4)(\R) \times U(3, 3)(\R)$. We write down the polynomials $\Tr(\chi_c^G f_{\alpha}, \pi_p) \in A^+$ for the representations that occur. The unramified characters $\eps, \eps_a, \eps'$ and the factor of similitudes have no influence on the form of the polynomials, so we leave them out. 

The computation of the compact traces on the representations $\pi_p = \Speh(3, 2)$ and $\pi_p = \Speh(2, 3)$ is done in the Figures~4 and~5. Recall that the computuation on $\speh(3, 2)$ is done via the segments of its dual representation $\Speh(2, 3)$. The Zelevinsky segments of the representation $\Speh(2, 3)$ are $\lbr -\tfrac 12, \tfrac 12, \tfrac 32 \rbr$ and $\lbr -\tfrac 32, -\tfrac 12, \tfrac 12\rbr$. To compute the compact traces we consider the line $\ell$ in $\Q^2$ of slope $\tfrac sn$ and consider the weights of non-crossing lattice paths. In our case there are two possible slopes, slope $\tfrac 12$ and slope $\tfrac 13$; these yield several different polynomials. 
 
\begin{figure}[h]
\begin{center}
\medskip
\input{speh23.pstricks}
\end{center}
\begin{caption}{The compact trace on the representation $\Speh(3, 2)$ with respect to the function $f_{6 \alpha 3}$. We have $\tfrac sn = \tfrac 12$, and $\ux_1 = \ell(-\tfrac 12)$, $\ux_2 = \ell(-\tfrac 32)$, $\uy_1 = \ell(\tfrac 32 + 1)$ and $\uy_2 = \ell(\tfrac 12 + 1)$. The permutation $w_0$ is equal to $(12)$. We see that there are two Dyck $2$-paths going from the points $\ux^{w_0}$ to the points $\uy$, and one of those paths is non-strict because it touches the line $\ell$. Therefore $\Dyck_\strict^+(\ux^{w_0}, \uy) = q^{-1/2\alpha - 1/2\alpha - 3/2 \alpha} = q^{-5/2\alpha }$ and $\Dyck^+(\ux^{w_0}, \uy) = q^{-5/2\alpha} + q^{-3/2\alpha}$. We conclude: $\Tr(\chi_c^{G(\qp)} f_{6 \alpha 3}, \Speh(3, 2)) = (-1)^{n-t} \sign(w_0) q^{\tfrac {s(n-s)}2 \alpha} q^{-5/2 \alpha} = -q^{2 \alpha}$. }
\end{caption}
\label{figure4}
\end{figure}

\begin{figure}[h]
\begin{center}
\medskip
\input{speh32.pstricks}
\end{center}
\begin{caption}{The compact trace on the representation $\pi_p = \Speh(3, 2)$ with respect to the function $f_{6 \alpha 2}$. We have $\tfrac sn = \tfrac 13$, and $\ux_1 = \ell(-\tfrac 12)$, $\ux_2 = \ell(-\tfrac 32)$, $\uy_1 = \ell(\tfrac 32 + 1)$ and $\uy_2 = \ell(\tfrac 12 + 1)$. The permutation $w_0$ is the trivial permutation. There is one Dyck $2$-path going from the points $\ux^{w_0}$ to the points $\uy$ and this $2$-path is strict. Therefore $\Dyck^+_\strict(\ux^{w_0}, \uy) = \Dyck^+(\ux^{w_0}, \uy) = q^{-1/2 \alpha - 3/2 \alpha} = q^{-\alpha}$.
We conclude: $\Tr(\chi_c^{G(\qp)} f_{6 \alpha 2}, \Speh(3,2)) = (-1)^{n-t} \sign(w_0) q^{\frac {s(n-s)}2} = -q^{3 \alpha}$.
}
\end{caption}
\label{figure5}
\end{figure}

In the illustrations we found that 
\begin{align*}
\Tr(\chi_c^{G(\qp)} f_{6 \alpha 3}, \Speh(3, 2)) & = -q^{-2 \alpha} \cr
\Tr(\chi_c^{G(\qp)} f_{6 \alpha 2}, \Speh(3, 2)) & = -q^{3 \alpha}.
\end{align*}
Using the duality and the computation in the figures, we find that
\begin{align*}
\Tr(\chi_c^{G(\qp)} f_{6 \alpha 3}, \Speh(2, 3)) & = q^{2 \alpha} + q^{3 \alpha} \cr
\Tr(\chi_c^{G(\qp)} f_{6 \alpha 2}, \Speh(2, 3)) & = q^{2\alpha}.  
\end{align*}
By drawing the picture, we see in a similar manner to the illustrations that
\begin{equation}\label{speh22}
\Tr(\chi_c^{G_4} f_{4 \alpha 2}, \Speh(2, 2)) = q^{3\alpha}. 
\end{equation}
and
\begin{align}\label{trivial62}
\Tr(\chi_c^{G_6} f_{6 \alpha 2}, \one_{G_6}) &= 1 + q^{\alpha} + q^{2\alpha} \cr
\Tr(\chi_c^{G_6} f_{6 \alpha 2}, \St_{G_6}) &= -(q^{\alpha} + q^{2 \alpha}) \cr
\Tr(\chi_c^{G_6} f_{6 \alpha 3}, \one_{G_6}) &= 1 + q^\alpha + 2q^{2\alpha} + q^{3\alpha} \cr 
\Tr(\chi_c^{G_6} f_{6 \alpha 3}, \St_{G_6}) &= -(1 + q^\alpha). 
\end{align}

The representations at $p$ occurring in the alternating sum of the cohomology of the basic stratum are (up to twists):
$$
\Speh(2, 3) \otimes \Speh(2, 3), \quad \Speh(2, 3) \otimes (\one_{G_3} \times \one_{G_3});
$$
\begin{align*} 
&\Speh(3, 2) \otimes \Speh(3, 2), \quad (\Speh(2, 2) \times \one_{G_2}) \otimes \Speh(3, 2) \quad (\one_{G_2} \times \Speh(2, 2)) \otimes \Speh(3, 2), \cr
&(\one_{G_2} \times \one_{G_2} \times \one_{G_2}) \otimes \Speh(3, 2); 
\end{align*}
$$
\St_{G_6} \otimes \St_{G_6}; 
$$
$$
\one_{G_6} \otimes \one_{G_6}.
$$

Let us ignore the factor of similitudes of the group $G(\qp)$. On the group $\Gl_2(\qp) \times \Gl_2(\qp)$ the function 
of Kottwitz is equal to $f_{6\alpha 2} \otimes f_{6\alpha 3} \in \cH_0(\Gl_n(\Qp)) \otimes \cH_0(\Gl_n(\qp))$. 
With the formulas we gave above the compact traces on the representations in this list are now all explicit.

\bibliographystyle{plain}
\bibliography{grotebib}

\end{document}